\documentclass[english]{article}
\usepackage[T1]{fontenc}
\usepackage[latin9]{inputenc}
\usepackage{babel}
\usepackage{array}
\usepackage{float}
\usepackage{units}
\usepackage{textcomp}
\usepackage{mathrsfs}
\usepackage{mathtools}
\usepackage{multirow}
\usepackage{amsmath}
\usepackage{amsthm}
\usepackage{amssymb}
\usepackage{stackrel}
\usepackage{graphicx}
\usepackage[authoryear]{natbib}
\usepackage[unicode=true]
 {hyperref}

\makeatletter

\newcommand{\noun}[1]{\textsc{#1}}
\providecommand{\tabularnewline}{\\}
\floatstyle{ruled}
\newfloat{algorithm}{tbp}{loa}
\providecommand{\algorithmname}{Algorithm}
\floatname{algorithm}{\protect\algorithmname}

\theoremstyle{plain}
\newtheorem{thm}{\protect\theoremname}[section]
\theoremstyle{plain}
\newtheorem{fact}[thm]{\protect\factname}
\theoremstyle{plain}
\newtheorem{lem}[thm]{\protect\lemmaname}

\usepackage{lineno}
\usepackage{amsthm}
\usepackage{thmtools}
\usepackage{thm-restate}

\newtheoremstyle{mystyle}%                % Name
  {}%                                     % Space above
  {}%                                     % Space below
  {\itshape}%                             % Body font
  {}%                                     % Indent amount
  {\bfseries}%                            % Theorem head font
  {.}%                                    % Punctuation after theorem head
  { }%                                    % Space after theorem head, ' ', or \newline
  {}%                                     % Theorem head spec (can be left empty, meaning `normal')

\theoremstyle{mystyle}

\@ifundefined{showcaptionsetup}{}{%
 \PassOptionsToPackage{caption=false}{subfig}}
\usepackage{subfig}
\makeatother

\providecommand{\factname}{Fact}
\providecommand{\lemmaname}{Lemma}
\providecommand{\theoremname}{Theorem}

\begin{document}
\title{The Asymptotics of Wide Remedians}
\author{Philip T.\ Labo\thanks{Author contact: \protect\href{http://mailto:plabo@alumni.stanford.edu}{plabo@alumni.stanford.edu}.}}
\maketitle
\begin{abstract}
The remedian uses a $k\times b$ matrix to approximate the median
of $n\leq b^{k}$ streaming input values by recursively replacing
buffers of $b$ values with their medians, thereby ignoring its $200\left(\nicefrac{\left\lceil \nicefrac{b}{2}\right\rceil }{b}\right)^{k}\%$
most extreme inputs. \citet{RB90} and \citet{CL93,CC05} study the
remedian's distribution as $k\rightarrow\infty$ and as $k,b\rightarrow\infty$.
The remedian's breakdown point vanishes as $k\rightarrow\infty$,
but approaches $\left(\nicefrac{1}{2}\right)^{k}$ as $b\rightarrow\infty$.
We study the remedian's robust-regime distribution as $b\rightarrow\infty$,
deriving a normal distribution for standardized (mean, median, remedian,
remedian rank) as $b\rightarrow\infty$, thereby illuminating the
remedian's accuracy in approximating the sample median. We derive
the asymptotic efficiency of the remedian relative to the mean and
the median. Finally, we discuss the estimation of more than one quantile
at once, proposing an asymptotic distribution for the random vector
that results when we apply remedian estimation in parallel to the
components of i.i.d.\ random vectors.
\end{abstract}

\section{Introduction \label{sec:Introduction}}

The word \emph{remedian} refers to a collection of algorithms that
take $b^{k}$ univariate input values and recursively whittle this
down to $b^{k-1},b^{k-2},\ldots,b^{0}=1$ values, where the last value
estimates some inner quantile. As originally designated, the algorithm
uses a $k\times b$ matrix to approximate the median of data as they
stream in (\citet{RB90}). Other formulations are possible:
\begin{itemize}
\item \citet{CL93} give a collection of algorithms for approximating any
pre-specified, non-extreme quantile. Section \ref{sec:Other-Quantiles}
studies this approach.
\item \citet{BCCCH00,CH13} describe a batch remedian algorithm that does
its work in-place on a single array of length $b^{k}$. 
\end{itemize}
The central idea---quantile estimation using recursion---predates
the word itself (\citet{T78,W78}). 

\begin{algorithm}
\caption{\noun{Insert $x$ into Remedian $\mathbf{R}$}}
\label{algo:insert}
\begin{description}
\item [{Note:}] This code works so long as $n+1<b^{k}$; we recommend using
$b$ odd
\item [{Inputs:}]~
\begin{enumerate}
\item Remedian $\mathbf{R}\in\widetilde{\mathbb{R}}{}^{k\times b}$ summarizing
$\mathbf{x}\in\mathbb{R}^{n}$
\item Data value $x\in\mathbb{R}$
\end{enumerate}
\item [{Output:}] Remedian $\mathbf{R}\in\tilde{\mathbb{R}}^{k\times b}$
summarizing $\left(\mathbf{x},x\right)\in\mathbb{R}^{n+1}$
\item [{Let}] $\hat{j}\leftarrow\min\left\{ j\in\left[b\right]:\mathbf{R}\left[1,j\right]=\textrm{NA}\right\} $
and $\mathbf{R}\left[1,\hat{j}\right]\leftarrow x$
\item [{For}] $i\in\left[k\right]$:
\begin{description}
\item [{If}] $\mathbf{R}\left[i,b\right]\ne\mathrm{NA}$:
\begin{description}
\item [{Let}] $\tilde{j}\leftarrow\min\left\{ j\in\left[b\right]:\mathbf{R}\left[i+1,j\right]=\textrm{NA}\right\} $
\item [{Let}] $\mathbf{R}\left[i+1,\tilde{j}\right]\leftarrow\textrm{median}\ \mathbf{R}\left[i,1\textrm{:}b\right]$
\item [{Let}] $\mathbf{R}\left[i,j\right]=\textrm{NA}$ for each $j\in\left[b\right]$
\end{description}
\end{description}
\end{description}
\end{algorithm}

We focus on approximating the median of univariate data as they stream
in. The streaming model of computation consists of a data source that
emits one data point at a time and a core algorithm that receives
each data point, carries out some computation, adjusts its internal
storage state, and---forgetting the data point---stands ready to
receive the next data point. The remedian---with internal storage
state the $k\times b$ matrix $\mathbf{R}$---is an elegant example
of streaming computation. Numerous organizations use streaming computation
to streamline the computation of realtime information and cut down
on storage costs.

\begin{algorithm}[t]
\caption{\noun{Query Remedian $\mathbf{R}$}}
\label{algo:query}
\begin{description}
\item [{Input:}] Remedian $\mathbf{R}\in\widetilde{\mathbb{R}}{}^{k\times b}$
\item [{Outputs:}]~
\begin{enumerate}
\item The remedian of the values summarized by $\mathbf{R}$
\item The number of values $n$ summarized by $\mathbf{R}$
\item The $n_{i}\in\mathbb{Z}_{b}$ such that $n=\sum_{i=1}^{k}n_{i-1}b^{i-1}$
\end{enumerate}
\item [{Let}] $\mathbf{r}\leftarrow\left(\right)$ and $\mathbf{w}\leftarrow\left(\right)$
\item [{Let}] $\mathbf{n}\leftarrow\mathbf{0}\in\mathbb{R}^{k}$ and $n\leftarrow0$
\item [{For}] $\left(i,j\right)\in\left[k\right]\times\left[b\right]$:
\begin{description}
\item [{If}] $\mathbf{R}\left[i,j\right]\ne\textrm{NA}$:
\begin{description}
\item [{Let}] $n\leftarrow n+b^{i-1}$
\item [{Let}] $n_{i-1}\leftarrow n_{i-1}+1$
\item [{Let}] $\mathbf{r}\leftarrow\mathbf{r}\oplus\mathbf{R}\left[i,j\right]$
\item [{Let}] $\mathbf{w}\leftarrow\mathbf{w}\oplus b^{i-1}$
\end{description}
\end{description}
\item [{Let}] $\hat{\mathbf{r}}\leftarrow\mathbf{r}\left[\mathrm{order}\left(\mathbf{r}\right)\right]$
be an ordered list of the elements in $\mathbf{r}$
\item [{Let}] $\hat{\mathbf{w}}\leftarrow\mathbf{w}\left[\mathrm{order}\left(\mathbf{r}\right)\right]$
be the corresponding list of weights
\item [{Let}] $\hat{m}\leftarrow\min\left\{ m\geq1:\sum_{l=1}^{m}\hat{w}_{l}\geq\nicefrac{n}{2}\right\} $
\item [{Return}] $\hat{\mathbf{r}}\left[\hat{m}\right]$, $n$, $\mathbf{n}$
\end{description}
\end{algorithm}

The remedian uses the following approach:
\begin{enumerate}
\item The remedian's internal storage state, the $k\times b$ matrix $\mathbf{R}$,
starts out empty. Each cell initially contains the symbol NA (Table
\ref{tab:notation}).
\item Insertion is recursive: Row 1 receives raw data values from the data
source, row 2 receives the medians of raw data values, row 3 receives
the medians of the medians of raw data values, \emph{etc}. When row
$i$ ($1\leq i<k$) is full, it places the median of its values into
the first empty cell in row $i+1$ and returns all of its cells to
the empty state (Algorithm \ref{algo:insert}).
\item The remedian estimate is the weighted median of the values in $\mathbf{R}$,
where we note that a value in row $i$ represents $b^{i-1}$ input
values (Algorithm \ref{algo:query}). In theoretical settings (as
here) the remedian estimate is the value the $k$th row sends on to
the nonexistent $\left(k+1\right)$st row, \emph{i.e.}, the remedian
returns this value right after the $b^{k}$th insertion.
\end{enumerate}
While the remedian ignores a positive fraction of its most extreme
inputs, its finite-sample breakdown point falls between those of the
mean and the median. Sending $k\rightarrow\infty$ ($b\rightarrow\infty$)
sends the remedian's breakdown point to zero (a limit between those
for the mean and median). See \S\ref{subsec:Breakdown-Point}.

\subsection{How Robust is the Remedian? \label{subsec:Breakdown-Point}}

Throughout we measure the robustness of an estimator using the breakdown
concept of \citet{H68}.\footnote{Some authors feel estimators with large breakdown points hide important
information in big data settings, favoring mixture modeling instead
(see \citet{HR09} \S11.1). That said, big data settings many times
warrant the lightweight tracking we propose here.} For finite $n$, the breakdown point of an estimator $T\coloneqq T\left(x_{1},x_{2},\ldots,x_{n}\right)$
is the smallest\footnote{The minimization occurs over the set of positions $i$ containing
corrupted data values.} fraction of the $x_{i}$ values that, when sent to infinity, corrupt
the value of $T$ (\citet{DH83}). For example, 
\begin{description}
\item [{Sample\ Mean}] The breakdown point of $\bar{x}\coloneqq\frac{1}{n}\sum_{i=1}^{n}x_{i}$
is $\frac{1}{n}\stackrel{n\rightarrow\infty}{\longrightarrow}0$ because
$\lim_{x_{i}\rightarrow\infty}\bar{x}=\infty$ for any $1\leq i\leq n$. 
\item [{Sample\ Median}] The breakdown point of 
\[
\mathrm{median}\left(\mathbf{x}\right)\coloneqq\left\{ \begin{array}{rl}
\frac{1}{2}\left(x_{\left(\nicefrac{n}{2}\right)}+x_{\left(\nicefrac{n}{2}+1\right)}\right) & \textrm{ if }n\in\left[0\right]_{2}\\
x_{\left(\nicefrac{\left(n+1\right)}{2}\right)} & \textrm{ if }n\in\left[1\right]_{2}
\end{array}\right.
\]
is $\frac{1}{n}\left\lceil \frac{n}{2}\right\rceil \stackrel{n\rightarrow\infty}{\longrightarrow}\frac{1}{2}$.
When $n\in\left[1\right]_{2}$ (or, $n\in\left[0\right]_{2}$), corrupting
the median requires sending at least $\frac{n+1}{2}=\left\lceil \frac{n}{2}\right\rceil $
(or, $\frac{n}{2}=\left\lceil \frac{n}{2}\right\rceil $) values to
infinity.
\item [{Remedian}] The breakdown point of the remedian when $n=b^{k}$
is 
\begin{equation}
\epsilon^{*}\left(k,b\right)\coloneqq\left(\left.\left\lceil \nicefrac{b}{2}\right\rceil \right/b\right)^{k}\longrightarrow\left\{ \begin{array}{ll}
0 & \textrm{ as }k\rightarrow\infty\\
2^{-k} & \textrm{ as }b\rightarrow\infty.
\end{array}\right.\label{eq:breakdown_lim}
\end{equation}
Putting $k=1$ gives the sample median. When $k=2$, corrupting the
remedian requires sending at least $\left\lceil \nicefrac{b}{2}\right\rceil $
second-row values to infinity, each of which requires sending at least
$\left\lceil \nicefrac{b}{2}\right\rceil $ first-row values to infinity.
When $k=2$, we must corrupt $\left\lceil \nicefrac{b}{2}\right\rceil \times\left\lceil \nicefrac{b}{2}\right\rceil $
values, and induction gives (\ref{eq:breakdown_lim}).
\end{description}
For $b\in\mathcal{B}_{3}$, $k\geq2$, and $n=b^{k}$, we have $\epsilon_{\mathrm{median}}^{*}>\epsilon_{\mathrm{remedian}}^{*}>\epsilon_{\mathrm{mean}}^{*}$.
Increasing remedian capacity by increasing the number of rows $k$
decimates robustness. In sections \ref{sec:Asymptotic-Normality}
and \ref{sec:Other-Quantiles} we study the setting in which $k$
remains fixed while $b\rightarrow\infty$.

\subsection{Previous Work \label{subsec:Previous-Work}}

\begin{table}
\hfill{}%
\begin{tabular}{r|l|l}
 & \noun{Statistics} & \noun{Computer Science}\tabularnewline
\hline 
\multirow{4}{*}{\noun{Remedian}} & \citet{T78} & \citet{W78}\tabularnewline
 & \citet{RB90} & \citet{BCCCH00}\tabularnewline
 & \citet{CL93} & \citet{CH13}\tabularnewline
 & \citet{CC05} & \tabularnewline
\hline 
 & \citet{JJS83} & \citet{SBAS04}\tabularnewline
\noun{Quantile} & \citet{JC85} & \citet{ACHPWY13}\tabularnewline
\noun{Estimation} & \citet{T93} & \citet{KLL16}\tabularnewline
\noun{on Data} & \citet{HM95} & \citet{DE19}\tabularnewline
\noun{Streams} & \citet{LLM03} & \citet{MRL19}\tabularnewline
 & \citet{CJLVW06} & \citet{CKLTV21}\tabularnewline
 & + a few more & + many, many more\tabularnewline
\end{tabular}\hfill{}

\caption{Articles about the remedian and quantile estimation on data streams.}

\label{tab:literature}
\end{table}

While the remedian comes out of the statistics literature, the computer
science literature has a great deal more to say about quantile estimation
on data streams (Table \ref{tab:literature}). \citet{RB90} show
that the remedian has breakdown point (\ref{eq:breakdown_lim}) and
that the remedian's estimate converges in probability to the population
median under weak conditions as $k\rightarrow\infty$. Furthermore,
they show that the standardized remedian converges to a non-Gaussian
limiting distribution as $k\rightarrow\infty$. \citet{CL93} and
\citet{CC05} continue this program, showing that the remedian's estimate
converges almost surely to the population median under weak conditions
as $b$ or $k$ become large (see \S\ref{subsec:Almost-Sure-Convergence})
and that the standardized remedian converges to normality as $b$
and $k$ both become large. Section \ref{subsec:Univariate-Normality}
shows that this latter result holds when $k$ remains fixed and $b\rightarrow\infty$.

As noted, \citet{BCCCH00,CH13} describe a batch remedian that does
its work in-place on an array of length $b^{k}$. They develop an
overly complex expression for the distribution of the remedian's rank
when $b=3$. While their derivation assumes $X_{1},X_{2},\ldots,X_{3^{k}}$
iid\emph{ $F$}, its complexity makes it impossible to pursue for
$b>3$. In section \ref{subsec:Univariate-Normality} we derive the
normal distribution of the remedian's standardized rank when $b\rightarrow\infty$.

If we compare the computer science and statistics literatures on the
topic of quantile estimation on data streams, the computer science
literature has more to say, but the statistics literature does not
lack for interesting results. \citet{P81} describes a minimax tree:
data percolate up from the leaves, undergoing alternating min and
max operations. The value that reaches the root converges in probability
to $F^{-1}\left(p\right)$ as the tree becomes tall. \citet{K86,D91,LB94}
design size-$\mathcal{O}\left(\sqrt{n}\right)$ algorithms that use
the confidences intervals of \citet{JJS83,Dv84}. \citet{T93} gives
an online algorithm based on the stochastic approximation algorithm
of \citet{RM51}. \citet{JC85,LLM03} use counting methods for approximating
$F^{-1}\left(p\right)$, which \citet{R87,R90,MBLL07} adapt to approximate
$\left\{ F^{-1}\left(p_{j}\right)\right\} _{j=1}^{m}$ (\emph{cf.}\ \S\ref{subsec:Any-Quantile}).
\citet{HM95,L24} recommend using histograms.

We describe some \emph{recent} results from the computer science literature.\footnote{See \citet{CY20} for a more detailed review of some of these algorithms.}
\citet{SBAS04} introduce Q-Digest, which assumes a universe of $N<\infty$
inputs and estimates rank to within $\epsilon n$ units using space-$\mathcal{O}\left(\nicefrac{1}{\epsilon}\log N\right)$
trees. \citet{ACHPWY13} describes \emph{compactors}, sequences of
buffers that, when full, send evenly- or oddly-indexed order statistics
on to the next buffer. \citet{KLL16} use these to design space-$\mathcal{O}\left(\epsilon^{-1}\log^{2}\log\nicefrac{1}{\delta\epsilon}\right)$
summaries that approximate rank to within $\epsilon n$ units with
probability $1-\delta$, and \citet{CKLTV21} use the same to design
space-
\[
\mathcal{O}\left(\frac{\log^{1.5}\left(\epsilon n\right)}{\epsilon}\sqrt{\log\left(\frac{\log\left(\epsilon n\right)}{\delta\epsilon}\right)}\right)
\]
summaries that approximate rank so that 
\[
\Pr\left(1-\epsilon\leq\frac{\textrm{estimated rank}}{\textrm{true rank}}\leq1+\epsilon\right)\geq1-\delta.
\]
\citet{MRL19} describe DDSketch, exponential histograms that guarantee
\[
1-\epsilon\leq\frac{\textrm{estimated rank}}{\textrm{true rank}}\leq1+\epsilon
\]
for space-$\mathcal{O}\left(\log n\right)$ summaries; see \citet{L24}.
\citet{DE19} design $t$-digest, a fixed-sized summary that uses
a scaling function to focus attention on extreme quantiles.

\subsection{Our Contributions \label{subsec:Our-Contributions}}

\begin{table}[t]
\makebox[\textwidth][c]{

\begin{tabular}{rll}
\hline 
\noun{Symbol} & \noun{Meaning} & \noun{Notes}\tabularnewline
\hline 
$\mathrm{NA}$ & unassigned or missing value & as in R (\citet{R22})\tabularnewline
$a\coloneqq b$ & $a$ is defined as $b$ & or $b\eqqcolon a$\tabularnewline
$a\leftarrow b$ & $a$ is assigned as $b$ & see Algorithms \ref{algo:insert} and \ref{algo:query}\tabularnewline
$\widetilde{\mathbb{R}}$ & $\mathbb{R}\cup\left\{ \textrm{NA}\right\} $ & augmented reals\tabularnewline
$\mathbb{Z}_{+}$ & $\left\{ 1,2,\ldots\right\} $ & positive integers\tabularnewline
$\mathcal{B}_{3}$ & $\left\{ 3,5,\ldots\right\} $ & possible remedian widths\tabularnewline
$\left[l\right]$ & $\left\{ 1,2,\ldots,l\right\} $ & $l\in\mathbb{Z}_{+}$\tabularnewline
$\left[l\right]_{m}$ & $\left\{ l\pm km:k=0,1,\ldots\right\} $ & $l\in\mathbb{Z},m\in\mathbb{Z}_{+}$\tabularnewline
$\mathrm{diag}\left(\mathbf{v}\right)$ & $m\times m$ matrix $a_{i,j}=\mathbf{1}_{\left\{ i=j\right\} }v_{i}$ & $\mathbf{v}\in\mathbb{R}^{m},m\in\mathbb{Z}_{+}$\tabularnewline
$\mathbf{v}\oplus x$ & $\left(v_{1},v_{2},\ldots,v_{K-1},v_{K},x\right)^{\mathsf{T}}$ & $\mathbf{v}\in\mathbb{R}^{K},x\in\mathbb{R},K\geq1$\tabularnewline
$\Gamma\left(x\right)$ & $\int_{0}^{\infty}t^{x-1}\exp\left(-t\right)dt$ & Gamma function\tabularnewline
$B\left(x,y\right)$ & $\int_{0}^{1}t^{x-1}\left(1-t\right)^{y-1}dt=\frac{\Gamma\left(x\right)\Gamma\left(y\right)}{\Gamma\left(x+y\right)}$ & Beta function\tabularnewline
$\mathbf{1}_{A}$ & 1 if $A$ is true, 0 otherwise & indicator function\tabularnewline
$f\circ g$ & $\left(f\circ g\right)\left(x\right)\coloneqq f\left(g\left(x\right)\right)$ & function composition\tabularnewline
$a_{n}\rightarrow b$ & $\lim_{n\rightarrow\infty}a_{n}=b$ & converging sequence\tabularnewline
$a_{n}\sim b_{n}$ & $\lim_{n\rightarrow\infty}\nicefrac{a_{n}}{b_{n}}=1$ & asymptotic equivalence\tabularnewline
$F\left(x\right)$ & $\Pr\left(X\leq x\right)$ & cumulative distribution function\tabularnewline
$X_{n}\implies X$ & $\Pr\left(X_{n}\leq x\right)\longrightarrow\Pr\left(X\leq x\right)^{*}$ & convergence in distribution\tabularnewline
$\mathscr{L}\left(X\right)$ & the distribution of RV $X$ & $\mathscr{L}$ stands for law\tabularnewline
$X\stackrel{\mathscr{L}}{=}Y$ & $\Pr\left(X\leq x\right)=\Pr\left(Y\leq x\right)$, $\forall x$ & equality in distribution\tabularnewline
LHS, RHS & \multicolumn{2}{l}{abbreviations for the left- and right-hand sides of an equation}\tabularnewline
iid or i.i.d. & \multicolumn{2}{l}{abbreviation for ``independent and identically-distributed''}\tabularnewline
\hline 
\end{tabular}

}

\caption{Conventions we use. We have convergence in distribution if $*$ holds
$\forall x$ at which $\Pr\left(X\protect\leq x\right)$ is continuous.}

\label{tab:notation}
\end{table}

Organizations frequently ingest large, potentially-partially-corrupted
data sets, wishing to cheaply, but robustly, monitor each data set's
central tendency. While the large-$k$ remedian uses $\mathcal{O}\left(\log n\right)$
space, it is not robust. The large-$b$ remedian, by contrast, is
robust, but uses $\mathcal{O}\left(\sqrt[k]{n}\right)$ space. Organizations
willing to spend more on storage for the sake of robustness should
opt for the large-$b$ remedian. That said, while the literature describes
the distributions of the large-$k$ and large-$k,b$ remedians, it
does not yet describe the same for the large-$b$ remedian (\S\ref{subsec:Previous-Work};
\citet{RB90,CL93,CC05}). We derive the large-$b$ distribution of
the standardized remedian estimate, which we use to derive the large-$b$
distribution of the standardized (mean, median, remedian, remedian
rank) vector. After that we turn to the large-$b$ distribution of
the standardized vector that results when we use $\ell\geq1$ remedians
in parallel to simultaneously approximate $\ell$ distinct, internal
quantiles.

Section \ref{sec:Preliminaries} presents several preliminaries: important
assumptions, a recursive formula for the remedian's distribution,
Bahadur's formula, and the remedian's almost sure convergence to the
population median. Sections \ref{sec:Asymptotic-Normality} and \ref{sec:Other-Quantiles}
present our core results. Section \ref{sec:Asymptotic-Normality}
shows that the standardized (mean, median, remedian, remedian rank)
vector approaches normality and derives the asymptotic relative efficiencies
between the mean, the median, and the remedian. Section \ref{sec:Other-Quantiles}
takes up a related topic, proposing an expression for the large-$b$
distribution of the standardized vector that results when we use $\ell\geq1$
remedians in parallel to approximate distinct, internal quantiles;
we prove the case $\ell=1$. Section \ref{sec:Conclusions-and-Discussion}
concludes, and Table \ref{tab:notation} defines some important notation.

\section{Preliminaries \label{sec:Preliminaries}}

While sections \ref{sec:Asymptotic-Normality} and \ref{sec:Other-Quantiles}
present our core results, this section describes several preliminaries,
namely, key assumptions, a recursive formula for the remedian's distribution,
Bahadur's formula, and the remedian's almost sure convergence to the
population median. 

\subsection{Assumptions}

Throughout this paper we assume that
\begin{equation}
X_{1},X_{2},\ldots,X_{n}\stackrel{\mathrm{iid}}{\sim}F,\label{eq:Xi}
\end{equation}
for CDF $F$. While our results are nonparametric (\citet{HWC14}),
they require a well-behaved $F$. For the remedian in section \ref{sec:Asymptotic-Normality}
our results require that
\begin{equation}
\begin{array}{c}
\exists a<\bar{\mu}<b\textrm{ such that }F'\textrm{ is continuous and positive on }\left(a,b\right)\\
\textrm{and }F'\textrm{ and }F''\textrm{ are bounded on \ensuremath{\left(a,b\right)},}
\end{array}\label{eq:AssumptionCOI_half}
\end{equation}
for $\bar{\mu}\coloneqq F^{-1}\left(\nicefrac{1}{2}\right)$ the population
median. More generally, for the remedian in section \ref{sec:Other-Quantiles}
our results require that
\begin{equation}
\begin{array}{c}
\exists a<\tilde{\mu}<b\textrm{ such that }F'\textrm{ is continuous and positive on }\left(a,b\right)\\
\textrm{and }F'\textrm{ and }F''\textrm{ are bounded on \ensuremath{\left(a,b\right)},}
\end{array}\label{eq:AssumptionCOI}
\end{equation}
for $\tilde{\mu}\coloneqq F^{-1}\left(\tilde{p}\right)$ the $\tilde{p}$th
population quantile. In general we require a continuous, positive,
bounded density $f\coloneqq F'$, and a bounded $f'=F''$, on a neighborhood
containing the estimand.\footnote{The requirement involving $F''$, atypical in this setting, comes
from \citet{B66}.}

\subsection{The Remedian's Distribution \label{subsec:Remedian's-Distribution}}

Let $X_{\left(1\right)}\leq X_{\left(2\right)}\leq\cdots\leq X_{\left(n\right)}$
give the order statistics, a sorted rearrangement of the i.i.d.\ values
in (\ref{eq:Xi}) (\citet{DN03}).\footnote{To obtain a unique sorted list, one can first sort by the indices,
and then by the values.} It is well-known that 
\begin{equation}
X_{\left(j\right)}\sim B_{j,n}\circ F,\label{eq:BjnDef}
\end{equation}
for $1\leq j\leq n$ and $B_{j,n}$ the $\mathrm{Beta}\left(j,n-j+1\right)$
CDF. In particular, with $m\geq1$, $b\coloneqq2m+1$, and $0<x<1$,
$\mathrm{median}\left\{ X_{1},X_{2},\ldots,X_{b}\right\} \sim\Psi_{b}^{\left(1\right)}\circ F$,
where
\begin{equation}
\Psi_{b}^{\left(1\right)}\left(x\right)\coloneqq B_{m+1,2m+1}\left(x\right)=\frac{b!}{m!^{2}}\int_{0}^{x}\left[y\left(1-y\right)\right]^{m}dy\label{eq:Psi(1)}
\end{equation}
is the $\mathrm{Beta}\left(m+1,m+1\right)$ CDF. (Table \ref{tab:notation}
defines $f\circ g$.) By the same reasoning, the remedian of $b^{2}$
values has distribution $\Psi_{b}^{\left(2\right)}\circ F\coloneqq\Psi_{b}^{\left(1\right)}\circ\Psi_{b}^{\left(1\right)}\circ F$,
so that, in general, the remedian of $b^{k}$ values has distribution
$G_{k,b}\coloneqq\Psi_{b}^{\left(k\right)}\circ F$, where
\begin{equation}
\Psi_{b}^{\left(0\right)}\left(x\right)\coloneqq x\textrm{ and }\Psi_{b}^{\left(k\right)}\left(x\right)\coloneqq\Psi_{b}^{\left(1\right)}\left(\Psi_{b}^{\left(k-1\right)}\left(x\right)\right),\label{eq:recursive}
\end{equation}
for $k\geq1$ (\citet{RB90}). Throughout we let $R_{k,b}$ represent
the remedian's rank among the i.i.d.\ $X_{1},X_{2},\ldots,X_{b^{k}}$,
so that, in summary, we have
\[
G_{k,b}\left(x\right)\coloneqq\Pr\left(X_{\left(R_{k,b}\right)}\leq x\right)=\Psi_{b}^{\left(k\right)}\left(F\left(x\right)\right).
\]

\subsection{Bahadur's Formula \label{subsec:Bahadur's-Formula}}

In the sequel Lemma \ref{lem:bahadur66} from \citet{B66} aids in
the proofs of results both well-known (Corollaries \ref{cor:bahadurcoras},
\ref{cor:bahadurcordist}, and \ref{cor:compMedian}) and new (Theorem
\ref{thm:rem_med_to_quad_normal}). One further suspects that Lemma
\ref{lem:bahadur66} will play a role in proving Conjectures \ref{conj:compRemedian}
and \ref{conj:quantRemedian}.

\begin{restatable}{lemma}{bahadur}

\label{lem:bahadur66}Fix $\tilde{p}\coloneqq1-\tilde{q}\in\left(0,1\right)$
and $\tilde{\mu}\coloneqq F^{-1}\left(\tilde{p}\right)$. Let $Y_{n}\coloneqq X_{\left(\left\lfloor 1+\left(n-1\right)\tilde{p}\right\rfloor \right)}$
and $Z_{n}\coloneqq\sum_{i=1}^{n}\mathbf{1}_{\left\{ X_{i}>\tilde{\mu}\right\} }$.
Then, under (\ref{eq:Xi}) and (\ref{eq:AssumptionCOI}), there is
a $C>0$ such that
\[
\Pr\left(\left|\sqrt{n}f\left(\tilde{\mu}\right)\left(Y_{n}-\tilde{\mu}\right)-\frac{Z_{n}-n\tilde{q}}{\sqrt{n}}\right|\geq\frac{C\log n}{n}\textrm{ infinitely often}\right)=0.
\]

\end{restatable}
\begin{proof}
This is the central result of the excellent \citet{B66}.
\end{proof}
Lemma \ref{lem:bahadur66} implies that $Y_{n}\longrightarrow\tilde{\mu}$
almost surely (a.s.), as $n\rightarrow\infty$.

\begin{restatable}{corollary}{bahadurcoras}

\label{cor:bahadurcoras}Under (\ref{eq:Xi}) and (\ref{eq:AssumptionCOI})
we have $\Pr\left(\lim_{n\rightarrow\infty}Y_{n}=\tilde{\mu}\right)=1$.

\end{restatable}
\begin{proof}
Lemma \ref{lem:bahadur66} implies that $f\left(\tilde{\mu}\right)\left(Y_{n}-\tilde{\mu}\right)\stackrel{\mathrm{a.s.}}{\sim}n^{-1}Z_{n}-\tilde{q}\longrightarrow0$
almost surely, as $n\rightarrow\infty$, where the limit uses Slutsky's
theorem and the SLLN.
\end{proof}
Lemma \ref{lem:bahadur66} further implies that $\sqrt{n}f\left(\tilde{\mu}\right)\left(Y_{n}-\tilde{\mu}\right)\implies\mathcal{N}\left(0,\,\tilde{p}\tilde{q}\right)$,
as $n\rightarrow\infty$.

\begin{restatable}{corollary}{bahadurcordist}

\label{cor:bahadurcordist}Under (\ref{eq:Xi}) and (\ref{eq:AssumptionCOI})
we have $\sqrt{n}f\left(\tilde{\mu}\right)\left(Y_{n}-\tilde{\mu}\right)\stackrel{n\rightarrow\infty}{\implies}\mathcal{N}\left(0,\,\tilde{p}\tilde{q}\right)$.

\end{restatable}
\begin{proof}
Lemma \ref{lem:bahadur66} implies that $\sqrt{n}f\left(\tilde{\mu}\right)\left(Y_{n}-\tilde{\mu}\right)\stackrel{\mathrm{a.s.}}{\sim}\frac{Z_{n}-n\tilde{q}}{\sqrt{n}}\implies\mathcal{N}\left(0,\,\tilde{p}\tilde{q}\right)$,
as $n\rightarrow\infty$, where the limit uses the CLT.
\end{proof}

\subsection{The Remedian Converges with Probability One \label{subsec:Almost-Sure-Convergence}}

While Corollary \ref{cor:bahadurcoras} implies that $\mathrm{median}\left\{ X_{1},X_{2},\ldots,X_{n}\right\} \longrightarrow\bar{\mu}$
almost surely, as $n\rightarrow\infty$, Theorem \ref{thm:asConv}
shows that $X_{\left(R_{k,b}\right)}\longrightarrow\bar{\mu}$ almost
surely, as $b\rightarrow\infty$.

\begin{restatable}{theorem}{asRemedConv}

\label{thm:asConv}Under (\ref{eq:Xi}) and (\ref{eq:AssumptionCOI})
we have $\Pr\left(\lim_{b\rightarrow\infty}X_{\left(R_{k,b}\right)}=\bar{\mu}\right)=1$.

\end{restatable}
\begin{proof}
Theorem 1 in
\begin{itemize}
\item \citet{RB90} proves convergence in probability as $k\rightarrow\infty$;
\item \citet{CL93} proves convergence with probability one as $k\rightarrow\infty$;
\item \citet{CC05} proves convergence with probability one as $n\rightarrow\infty$
along any sequence of odd numbers raised to a positive integer power.
\end{itemize}
\end{proof}
Theorem \ref{thm:asConv} in turn implies that $b^{-k}R_{k,b}\longrightarrow\nicefrac{1}{2}$
almost surely, as $b\rightarrow\infty$.

\begin{restatable}{corollary}{asRankConv}

Under (\ref{eq:Xi}) and (\ref{eq:AssumptionCOI}) we have $\Pr\left(\lim_{b\rightarrow\infty}b^{-k}R_{k,b}=\nicefrac{1}{2}\right)=1$.

\end{restatable}
\begin{proof}
Note that $b^{-1}R_{1,b}\equiv\frac{b+1}{2b}\stackrel{b\rightarrow\infty}{\longrightarrow}\nicefrac{1}{2}$.
Assume that $k\geq2$. By contradiction assume that $\Pr\left(\lim_{b\rightarrow\infty}b^{-k}R_{k,b}=\nicefrac{1}{2}\right)<1$
\begin{align}
\implies & \exists\epsilon>0\textrm{ such that }\Pr\left(\left|b^{-k}R_{k,b}-\nicefrac{1}{2}\right|>\epsilon\textrm{ infinitely often}\right)>0\\
\implies & \exists\epsilon'>0\textrm{ such that }\Pr\left(\left|X_{\left(R_{k,b}\right)}-\bar{\mu}\right|>\epsilon'\textrm{ infinitely often}\right)>0,\label{eq:useAssum}
\end{align}
where (\ref{eq:useAssum}) uses (\ref{eq:AssumptionCOI_half}). This
contradicts Theorem \ref{thm:asConv}, completing the proof.
\end{proof}

\section{Asymptotic Normality \label{sec:Asymptotic-Normality}}

While section \ref{subsec:Almost-Sure-Convergence} considers the
almost sure convergence of $X_{\left(R_{k,b}\right)}$ and $b^{-k}R_{k,b}$,
section \ref{subsec:Univariate-Normality} shows that standardized
$X_{\left(R_{k,b}\right)}$ and $b^{-k}R_{k,b}$ converge to normality.
Section \ref{subsec:Joint-Asymptotic-Normality} further proves the
asymptotic normality of the standardized (mean, median, remedian,
remedian rank) vector, and section \ref{subsec:Asymptotic-Relative-Efficiency}
studies the asymptotic efficiency of the remedian relative to the
mean and the median.

\subsection{Univariate Asymptotic Normality \label{subsec:Univariate-Normality}}

We start with the asymptotic normality of the standardized remedian.
While Corollary \ref{cor:bahadurcordist} implies that 
\[
2b^{\nicefrac{k}{2}}f\left(\bar{\mu}\right)\left(X_{\left(\frac{b^{k}+1}{2}\right)}-\bar{\mu}\right)\stackrel[\infty]{b}{\Longrightarrow}\mathcal{N}\left(0,1\right),
\]
replacing $\frac{b^{k}+1}{2}$ with $R_{k,b}$ on the LHS replaces
1 with $\left(\nicefrac{\pi}{2}\right)^{k-1}$ on the RHS. 

\begin{restatable}{theorem}{remToNormal}

\label{thm:rem_to_normal}Under (\ref{eq:Xi}) and (\ref{eq:AssumptionCOI_half})
$2b^{\nicefrac{k}{2}}f\left(\bar{\mu}\right)\left(X_{\left(R_{k,b}\right)}-\bar{\mu}\right)\stackrel[\infty]{b}{\Longrightarrow}\mathcal{N}\left(0,\left(\nicefrac{\pi}{2}\right)^{k-1}\right)$.

\end{restatable}
\begin{proof}
While the statement of Theorem 2 in \citet{CC05} requires that $b,k\rightarrow\infty$
as $n\coloneqq b^{k}\rightarrow\infty$, their proof works when just
$b\rightarrow\infty$. Appendix \ref{sec:CC05-Arg} presents a slightly
modified version of \citet{CC05}'s argument.
\end{proof}
Our asymptotic results reflect that the remedian \emph{is} the median
when $k=1$. After that, each additional remedian buffer increases
the asymptotic variance of the remedian over that of the median by
an additional factor of $\nicefrac{\pi}{2}\approx1.57$ to account
for the increasing noisiness of $R_{k,b}$ as $k$ grows (\emph{cf.}\ Theorem
\ref{thm:rem_rank_to_normal}).

\subsubsection{The Iterated Remedian \label{subsec:The-Iterated-Remedian}}

Say we have, not one remedian matrix (as above), but $m\geq1$ remedian
matrices $\left\{ \mathbf{R}_{j}\right\} _{j=1}^{m}$, where $\mathbf{R}_{j}\in\widetilde{\mathbb{R}}^{k_{j}\times b_{t,j}}$,
for $\left(k_{j},b_{t,j}\right)\in\mathbb{Z}_{+}\times\mathcal{B}_{3}$.
What is the distribution of the estimate that results when we string
these together, so that the output of $\mathbf{R}_{j}$ becomes the
input of $\mathbf{R}_{j+1}$, $1\leq j<m$? In particular, we imagine
that:
\begin{enumerate}
\item For $1\leq i\leq\prod_{j=1}^{m}b_{t,j}^{k_{j}}$, the i.i.d.\ inputs
$X_{i}$ flow into the first row of $\mathbf{R}_{1}$;
\item For $1\leq j<m$, the outputs of $\mathbf{R}_{j}$ flow into the first
row of $\mathbf{R}_{j+1}$; and
\item The output of $\mathbf{R}_{m}$, which we call $Y_{\mathbf{k},\mathbf{b}_{t}}$,
gives us our final estimate.
\end{enumerate}
The following lemma shows that, if $b_{t,i}\sim C_{i,j}b_{t,j}\rightarrow\infty$,
for $1\leq i,j\leq m$, then 

\[
2f\left(\bar{\mu}\right)\left(\prod_{j=1}^{m}b_{t,j}^{\nicefrac{k_{j}}{2}}\right)\left\{ Y_{\mathbf{k},\mathbf{b}_{t}}-\bar{\mu}\right\} \stackrel[\infty]{t}{\implies}\mathcal{N}\left(0,\left(\nicefrac{\pi}{2}\right)^{\sum_{j=1}^{m}k_{j}-1}\right).
\]

\begin{restatable}{lemma}{iterRemToNormal}

\label{lem:iter_rem_to_normal}Fix $m\geq1$, $\mathbf{k}\in\mathbb{Z}_{+}^{m}$,
and $\mathbf{b}_{t}\in\mathcal{B}_{3}^{m}$ such that $\lim_{t\rightarrow\infty}b_{t,j}=\infty$
and $\lim_{t\rightarrow\infty}\nicefrac{b_{t,i}}{b_{t,j}}\eqqcolon C_{i,j}\in\left(0,\infty\right)$,
for $1\leq i,j\leq m$. Under (\ref{eq:Xi}) and (\ref{eq:AssumptionCOI_half})
we have
\begin{equation}
\Psi_{b_{t,m}}^{\left(k_{m}\right)}\circ\Psi_{b_{t,m-1}}^{\left(k_{m-1}\right)}\circ\cdots\circ\Psi_{b_{t,1}}^{\left(k_{1}\right)}\circ F\left(\bar{\mu}+\frac{\left(\nicefrac{\pi}{2}\right)^{\frac{\sum_{j=1}^{m}k_{j}-1}{2}}y}{2f\left(\bar{\mu}\right)\prod_{j=1}^{m}b_{t,j}^{\nicefrac{k_{j}}{2}}}\right)\stackrel{t\rightarrow\infty}{\longrightarrow}\Phi\left(y\right).\label{eq:iterLim}
\end{equation}

\end{restatable}
\begin{proof}
The proof, which uses Theorem \ref{thm:rem_to_normal}, appears in
Appendix \ref{sec:Iterated_Remedians}.
\end{proof}
A batch interpretation of Lemma \ref{lem:iter_rem_to_normal} involves
a random tensor $\left(X_{i_{1},i_{2},\ldots,i_{m}}\right)$---the
$X_{i_{1},i_{2},\ldots,i_{m}}$ are i.i.d.\ $F$---taking on values
in $\mathbb{R}^{b_{t,1}\times b_{t,2}\times\cdots\times b_{t,m}}$.
Letting
\begin{align*}
X_{i_{2},i_{3,}\ldots,i_{m}}^{\left(1\right)} & \coloneqq\mathrm{median}\left\{ X_{1,i_{2},\ldots,i_{m}},X_{2,i_{2},\ldots,i_{m}},\ldots,X_{b_{t,1},i_{2},\ldots,i_{m}}\right\} \\
X_{i_{3},i_{4},\ldots,i_{m}}^{\left(2\right)} & \coloneqq\mathrm{median}\left\{ X_{1,i_{3},\ldots,i_{m}}^{\left(1\right)},X_{2,i_{3},\ldots,i_{m}}^{\left(1\right)},\ldots,X_{b_{t,2},i_{3},\ldots,i_{m}}^{\left(1\right)}\right\} \\
 & \:\:\:\vdots\qquad\qquad\qquad\qquad\vdots\qquad\qquad\qquad\qquad\vdots\\
X^{\left(m\right)} & \coloneqq\mathrm{median}\left\{ X_{1}^{\left(m-1\right)},X_{2}^{\left(m-1\right)},\ldots,X_{b_{t,m}}^{\left(m-1\right)}\right\} ,
\end{align*}
for $\left(i_{j},i_{j+1},\ldots,i_{m}\right)\in\left[b_{t,j}\right]\times\left[b_{t,j+1}\right]\times\cdots\times\left[b_{t,m}\right]$,
we note that $X^{\left(m\right)}\stackrel{\mathscr{L}}{=}Y_{\mathbf{1}_{m},\mathbf{b}_{t}}$,
for $\mathbf{1}_{m}$ the vector of ones in $\mathbb{R}^{m}$, so
that Lemma \ref{lem:iter_rem_to_normal} applies. Note finally that,
by (\ref{eq:recursive}), Lemma \ref{lem:iter_rem_to_normal} with
$\mathbf{1}_{m}$ in place of $\mathbf{k}\in\mathbb{Z}_{+}^{m}$ retains
its full generality. 

Theorem \ref{thm:rem_to_normal} yields to the same batch interpretation
(put $b_{t,j}=b$), and the iterated remedian has breakdown point
\[
\prod_{j=1}^{m}\left(\left.\left\lceil \nicefrac{b_{t,j}}{2}\right\rceil \right/b_{t,j}\right)^{k_{j}},
\]
generalizing (\ref{eq:breakdown_lim}). Additional matrices, like
additional rows, decimate robustness.

\subsubsection{The Remedian's Rank \label{subsec:The-Remedian's-Rank}}

\begin{table}
\hfill{}%
\begin{tabular}{r|rrrrrrrrr}
$r$ & 1 & 2 & 3 & 4 & 5 & 6 & 7 & 8 & 9\tabularnewline
\hline 
$14\times\Pr\left(R_{2,3}=r\right)$ & 0 & 0 & 0 & 3 & 8 & 3 & 0 & 0 & 0\tabularnewline
\end{tabular}\hfill{}

\caption{$\mathscr{L}\left(R_{2,3}\right)$ from \citet{T78}. Values assume
continuous $F$. Note that $\left\lceil \nicefrac{3}{2}\right\rceil ^{2}=4$
implies that the possible ranks are $\left\{ 4,5,6\right\} $; see
section \ref{subsec:Breakdown-Point}.}

\label{tab:LR23}
\end{table}

Interest in $\mathscr{L}\left(R_{k,b}\right)$ dates to \citet{T78,RB90}.
\citet{T78} in particular uses combinatorics to derive $\mathscr{L}\left(R_{2,3}\right)$
(Table \ref{tab:LR23}). \citet{CH13} generalize this, deriving a
complex expression for $\mathscr{L}\left(R_{k,3}\right)$. We generalize
this in the direction of wide remedians, deriving a simple expression
for $\lim_{b\rightarrow\infty}\mathscr{L}\left(2b^{\nicefrac{k}{2}}\left(\nicefrac{R_{k,b}}{b^{k}}-\nicefrac{1}{2}\right)\right)$,
which holds for any $k\geq1$.

\begin{restatable}{theorem}{remedRankToNormal}

\label{thm:rem_rank_to_normal}Under (\ref{eq:Xi}) and (\ref{eq:AssumptionCOI_half})
$2b^{\nicefrac{k}{2}}\left(\nicefrac{R_{k,b}}{b^{k}}-\nicefrac{1}{2}\right)\stackrel[\infty]{b}{\Longrightarrow}\mathcal{N}\left(0,\left(\nicefrac{\pi}{2}\right)^{k-1}-1\right)$.

\end{restatable}
\begin{proof}
The proof, which uses Theorem \ref{thm:rem_to_normal}, appears in
Appendix \ref{sec:Rank_to_Normal}.
\end{proof}
Noting that $R_{1,b}\equiv\frac{b+1}{2}$ confirms the result for
$k=1$ (\emph{cf.}\ Theorem \ref{thm:rem_to_normal}). For $b$ large,
we have
\[
R_{k,b}\stackrel{\cdot}{\sim}\mathcal{N}\left(\frac{b^{k}+1}{2},\frac{b^{k}}{4}\left[\left(\nicefrac{\pi}{2}\right)^{k-1}-1\right]\right),
\]
so that, in this setting, the mean and variance of $R_{k,b}$ grow
exponentially in $k$, as expected. Note that, given the ranks of
the $X_{i}$, the remedian's rank follows deterministically; \emph{i.e.},
$R_{k,b}$ depends on $\left\{ X_{i}\right\} _{i=1}^{n}$ only through
$\left\{ \mathrm{rank}\left(X_{i}\right)\right\} _{i=1}^{n}$. The
order statistics provide no additional information: $\mathscr{L}\left(R_{k,b}\right)$
does not depend on $\bar{\mu}$ (\emph{cf.}\ Theorem \ref{thm:rem_to_normal}).

Finally, while \citet{CH13} study $\mathscr{L}\left(\left|R_{k,b}-\frac{b^{k}+1}{2}\right|\right)$
(see Tables 3 through 5), we note that:

\begin{restatable}{definition}{halfnormaldef}

\label{def:halfNormal}For $X\sim\mathcal{N}\left(0,\sigma^{2}\right)$,
$\left|X\right|$ is half-normal with $\mathbb{E}\left|X\right|=\sigma\sqrt{\nicefrac{2}{\pi}}$
and $\mathrm{Var}\left(\left|X\right|\right)=\sigma^{2}\left(1-\nicefrac{2}{\pi}\right)$
(\citet{LNN61}), \emph{i.e.}, $\left|X\right|\sim\mathcal{HN}\left(\sigma^{2}\right)$.

\end{restatable}

\begin{restatable}{corollary}{halfnormal}

Under (\ref{eq:Xi}) and (\ref{eq:AssumptionCOI_half}) $2b^{\nicefrac{k}{2}}\left|\nicefrac{R_{k,b}}{b^{k}}-\nicefrac{1}{2}\right|\stackrel[\infty]{b}{\Longrightarrow}\mathcal{HN}\left(\left(\nicefrac{\pi}{2}\right)^{k-1}-1\right)$,
which implies that, as $b\rightarrow\infty$,
\begin{align}
\mathbb{E}\left|R_{k,b}-\frac{b^{k}+1}{2}\right| & \sim\sqrt{\frac{b^{k}}{2\pi}\left[\left(\nicefrac{\pi}{2}\right)^{k-1}-1\right]}\label{eq:meanDevR}\\
\mathrm{Var}\left(\left|R_{k,b}-\frac{b^{k}+1}{2}\right|\right) & \sim\frac{b^{k}\left(1-\nicefrac{2}{\pi}\right)}{4}\left[\left(\nicefrac{\pi}{2}\right)^{k-1}-1\right].\label{eq:sdDevR}
\end{align}

\end{restatable}

While our results assume $b$ large (Figure \ref{fig:meanSDR}), they
complement and expand on those in \citet{CH13}, who focus on $\left(k,b\right)\in\left\{ 2,3,\ldots,9\right\} \times\left\{ 3,5,7\right\} $.

\begin{figure}
\hfill{}\includegraphics[scale=0.5]{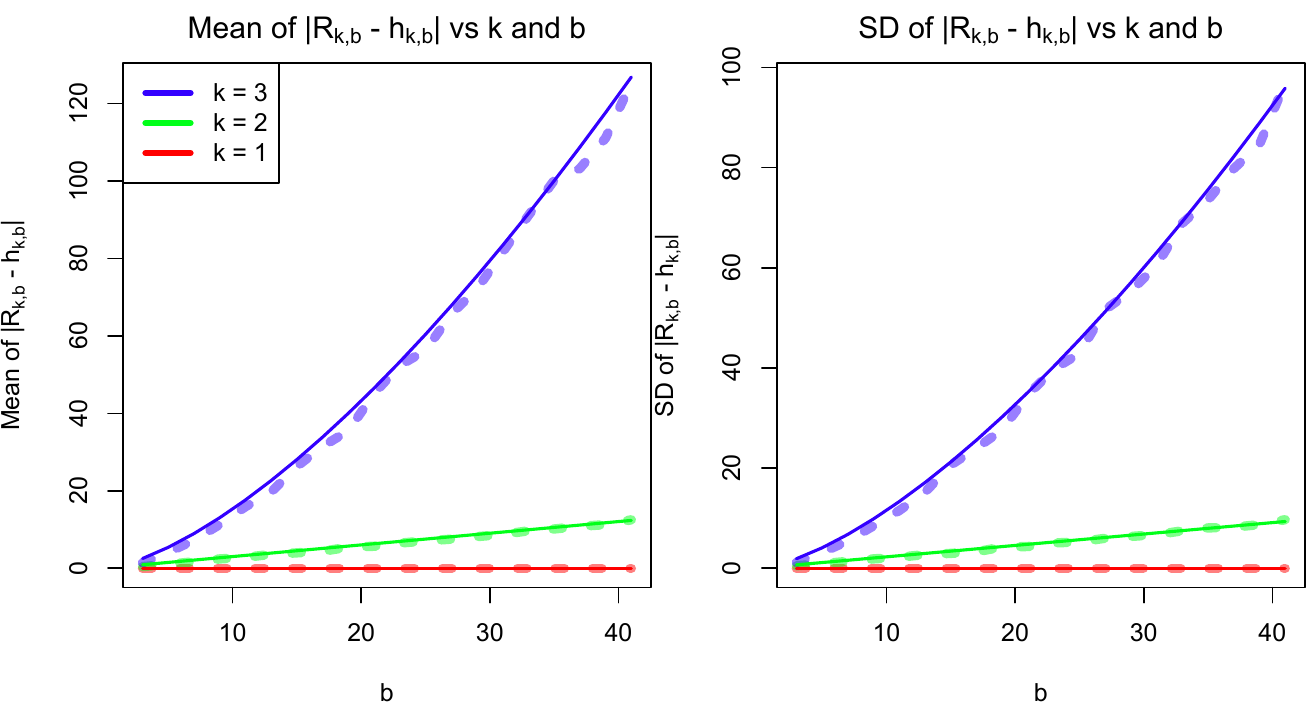}\hfill{}

\caption{Mean and standard deviation of $\left|R_{k,b}-h_{k,b}\right|$, where
$h_{k,b}\protect\coloneqq\nicefrac{\left(b^{k}+1\right)}{2}$. Solid
lines use (\ref{eq:meanDevR})--(\ref{eq:sdDevR}) while dotted lines
give the mean and standard deviation of $\min\left(b^{k}!,1000\right)$
simulated $\left|R_{k,b}-h_{k,b}\right|$ values (using continuous
$F$). We plot $\left(k,b\right)\in\left\{ 1,2,3\right\} \times\left\{ 3,5,\ldots,41\right\} $.
Note that $k=1$ gives the median itself.}

\label{fig:meanSDR}
\end{figure}

\subsection{Multivariate Asymptotic Normality \label{subsec:Joint-Asymptotic-Normality}}

Above we consider the univariate convergence of the standardized median
(\S\ref{subsec:Bahadur's-Formula}), remedian (\S\ref{subsec:Univariate-Normality}),
and remedian rank (\S\ref{subsec:The-Remedian's-Rank}). The central
limit theorem (CLT) gives the univariate convergence of the standardized
mean $\bar{X}_{n}\coloneqq n^{-1}\sum_{i=1}^{n}X_{i}$. We now show
that the standardized (mean, median, remedian, remedian rank) vector
approaches quadrivariate normality. In so doing we derive the asymptotic
correlations between all pairs of standardized variables.

\begin{restatable}{theorem}{remMedToQuadNormal}

\label{thm:rem_med_to_quad_normal}Under (\ref{eq:Xi}) and (\ref{eq:AssumptionCOI_half})
and $0<\sigma^{2}\coloneqq\mathrm{Var}\left(X_{1}\right)<\infty$
we have
\begin{equation}
b^{\nicefrac{k}{2}}\left(\begin{array}{r}
\bar{X}_{b^{k}}-\mu\\
X_{\left(h_{k,b}\right)}-\bar{\mu}\\
X_{\left(R_{k,b}\right)}-\bar{\mu}\\
b^{-k}R_{k,b}-\nicefrac{1}{2}
\end{array}\right)\stackrel[\infty]{b}{\Longrightarrow}\mathcal{N}_{4}\left(\left(\begin{array}{c}
0\\
0\\
0\\
0
\end{array}\right),\mathbf{D}\left(\begin{array}{cccc}
\sigma^{2} & \eta & \eta & 0\\
\eta & 1 & 1 & 0\\
\eta & 1 & \tau_{k}^{2} & \bar{\tau}_{k}^{2}\\
0 & 0 & \bar{\tau}_{k}^{2} & \bar{\tau}_{k}^{2}
\end{array}\right)\mathbf{D}\right),\label{eq:quad-var-lim}
\end{equation}
where $\mu\coloneqq\mathbb{E}X_{1}$, $h_{k,b}\coloneqq\frac{b^{k}+1}{2}$,
$\tau_{k}^{2}\coloneqq\left(\nicefrac{\pi}{2}\right)^{k-1}\eqqcolon1+\bar{\tau}_{k}^{2}$,
$\eta\coloneqq\mathbb{E}\left|X_{1}-\bar{\mu}\right|$, and $\mathbf{D}\coloneqq\mathrm{diag}\left(1,2f\left(\bar{\mu}\right),2f\left(\bar{\mu}\right),2\right)^{-1}$.

\end{restatable}
\begin{proof}
The proof, which uses Lemma \ref{lem:bahadur66}, Theorem \ref{thm:rem_rank_to_normal},
and the Cram\'er-Wold device, appears in Appendix \ref{sec:Multivariate-Convergence}.
\end{proof}
\begin{figure}
\hfill{}\subfloat{\includegraphics[scale=0.5]{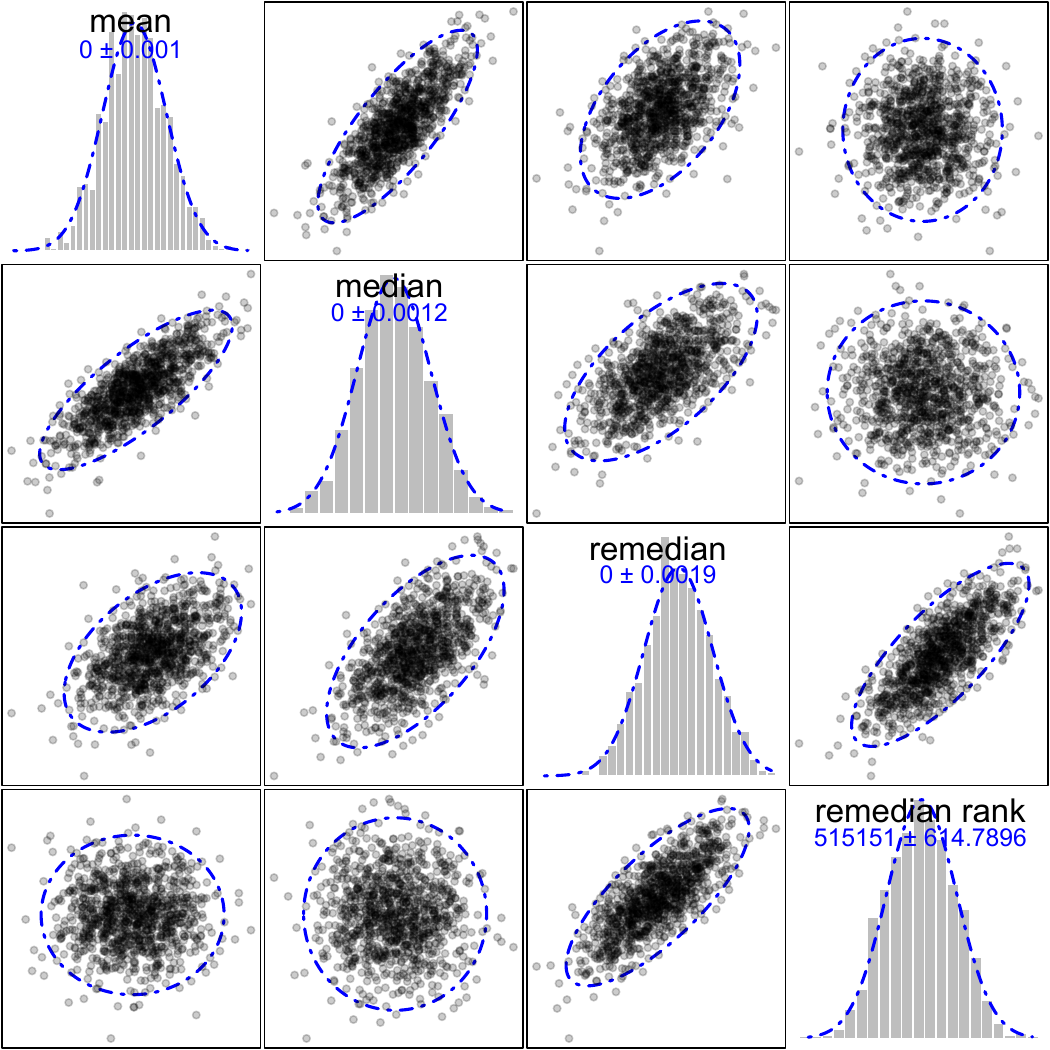}}\hfill{}

\hfill{}\subfloat{\includegraphics[scale=0.5]{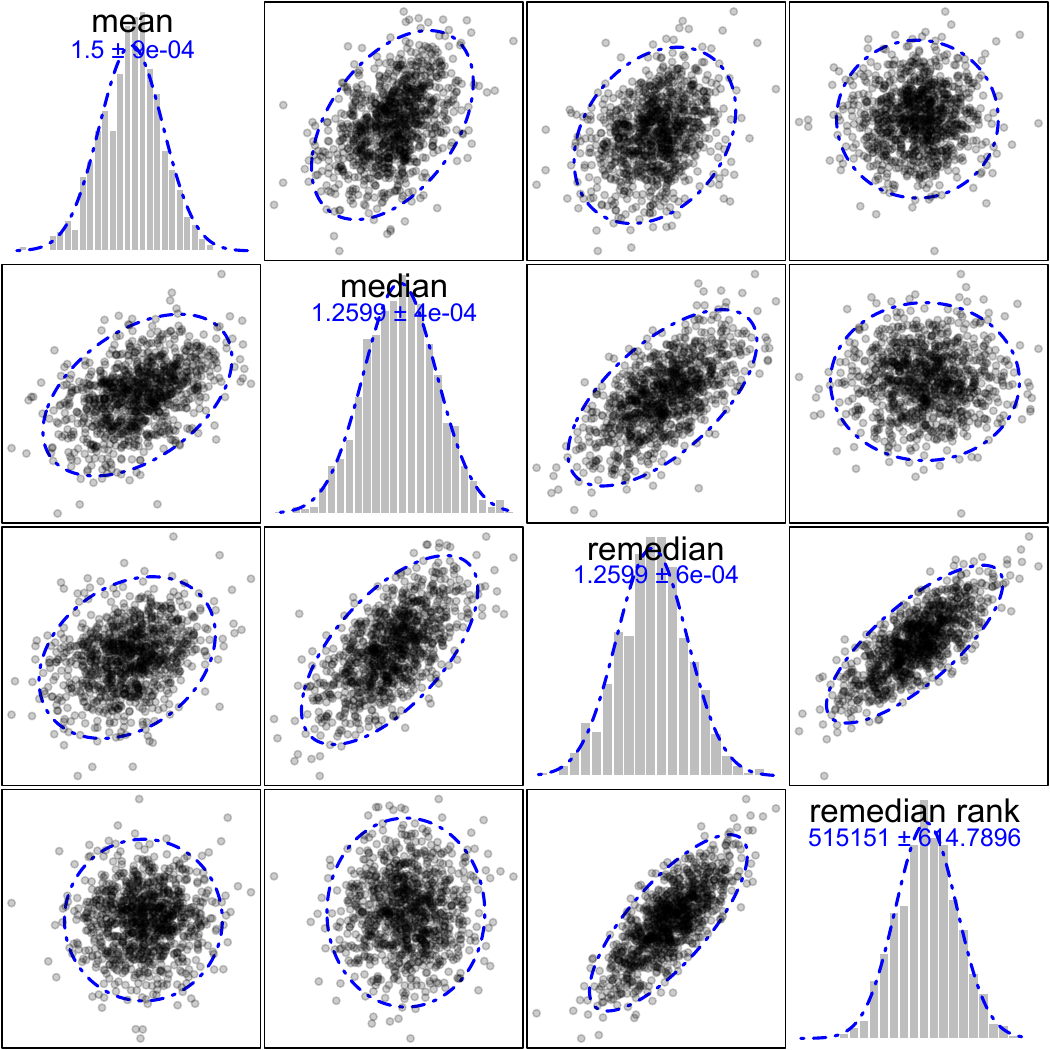}}\hfill{}

\caption{Theorem \ref{thm:rem_med_to_quad_normal} pairs ($k=3$, $b=101$).
Curves: normal densities, 95\% ellipses. Text: $\mathrm{mean}\pm\mathrm{sd}$.
Top: $X_{i}\stackrel{\mathrm{iid}}{\sim}\mathcal{N}\left(0,1\right)$.
Bottom: $X_{i}\stackrel{\mathrm{iid}}{\sim}\mathrm{Pareto}\left(1,3\right)$.}

\label{fig:ThmMultiNormal}
\end{figure}

Figure \ref{fig:ThmMultiNormal} shows simulated mean, median, remedian,
and remedian rank values when $k=3$, $b=101$, and $F\in\left\{ \mathcal{N}\left(0,1\right),\mathrm{Pareto}\left(1,3\right)\right\} $,\footnote{\label{fn:Pareto}For positive $\alpha$ and $\beta$ the $\mathrm{Pareto}\left(\alpha,\beta\right)$
distribution has $F\left(x\right)=(1-(\nicefrac{\alpha}{x})^{\beta})\mathbf{1}_{\left\{ x>\alpha\right\} }$
and
\[
\bar{\mu}=\alpha\sqrt[\beta]{2},\quad\mu=\left\{ \begin{array}{rl}
\infty & \textrm{if }\beta\leq1\\
\frac{\alpha\beta}{\beta-1} & \textrm{if }\beta>1,
\end{array}\right.\textrm{and}\quad\sigma^{2}=\left\{ \begin{array}{rl}
\infty & \textrm{if }\beta\leq2\\
\frac{\alpha^{2}\beta}{\left(\beta-1\right)^{2}\left(\beta-2\right)} & \textrm{if }\beta>2.
\end{array}\right.
\]
\vspace{-20pt}} illustrating Theorems \ref{thm:rem_to_normal}, \ref{thm:rem_rank_to_normal},
and \ref{thm:rem_med_to_quad_normal}. Simulated (pairs of) values
(in gray) approximately follow their derived distributions (in blue).
Replacing $\mathrm{Pareto}\left(1,3\right)$ with, \emph{e.g.}, $\mathrm{Pareto}\left(1,1\right)$
(which has infinite mean and variance) produces plots with outlying
$\bar{X}_{n}$, but normally-distributed medians, remedians, and remedian
ranks. The remedian with $\left(k,b\right)=\left(3,101\right)$ uses
74.3\% of its central inputs (\S\ref{subsec:Breakdown-Point}). Theorem
\ref{thm:rem_med_to_quad_normal} holds for its last three variables
when $\mathbb{E}\left|X_{1}\right|=\infty$ while the streaming computation
of $\bar{X}_{n}$ stores only two numbers: $n$ and $\sum_{i=1}^{n}X_{i}$.

The standardized variables above have the following limiting correlations:

\begin{restatable}{corollary}{correlation}

\label{cor:4by4correlation}If $p_{k}\coloneqq\left(\nicefrac{2}{\pi}\right)^{k-1}$
and $\mathbf{L}_{k}^{\mathsf{T}}\coloneqq\left(\textrm{\textsc{mean}},\textrm{\textsc{med}},\textrm{\textsc{remed}},\textrm{\textsc{remedRank}}\right)$
has the large-$b$ distribution on the RHS of (\ref{eq:quad-var-lim}),
then $\mathbf{L}_{k}$ has correlation matrix
\[
\mathbf{P}_{k}\coloneqq\left(\begin{array}{cccc}
1 & \frac{\eta}{\sigma} & \frac{\eta\sqrt{p_{k}}}{\sigma} & 0\\
\frac{\eta}{\sigma} & 1 & \sqrt{p_{k}} & 0\\
\frac{\eta\sqrt{p_{k}}}{\sigma} & \sqrt{p_{k}} & 1 & \sqrt{1-p_{k}}\\
0 & 0 & \sqrt{1-p_{k}} & 1
\end{array}\right)\stackrel[\infty]{k}{\longrightarrow}\left(\begin{array}{cccc}
1 & \frac{\eta}{\sigma} & 0 & 0\\
\frac{\eta}{\sigma} & 1 & 0 & 0\\
0 & 0 & 1 & 1\\
0 & 0 & 1 & 1
\end{array}\right).
\]
The normality of the $\mathbf{L}_{k}$, $k\geq1$, then implies the
following large-$b$ dependencies:
\[
\begin{array}{rrrr}
 &  & \underline{k\textrm{ finite}} & \underline{k\rightarrow\infty}\\
\textrm{\textsc{mean}} & \textrm{\textsc{med}} & \textrm{dependent} & \textrm{dependent}\\
\textrm{\textsc{mean}} & \textrm{\textsc{remed}} & \textrm{weakly dependent} & \textrm{independent}\\
\textrm{\textsc{mean}} & \textrm{\textsc{remedRank}} & \textrm{independent} & \textrm{independent}\\
\textrm{\textsc{med}} & \textrm{\textsc{remed}} & \textrm{weakly dependent} & \textrm{independent}\\
\textrm{\textsc{med}} & \textrm{\textsc{remedRank}} & \textrm{independent} & \textrm{\textrm{independent}}\\
\textrm{\textsc{remed}} & \textrm{\textsc{remedRank}} & \textrm{strongly dependent} & \textrm{linearly dependent}
\end{array}
\]

\end{restatable}

\begin{restatable}{example}{corEgs}

If $X_{i}\stackrel{\mathrm{iid}}{\sim}\mathcal{N}\left(\mu,\sigma^{2}\right)$,
Definition \ref{def:halfNormal} gives $\eta=\sigma\sqrt{\nicefrac{2}{\pi}}$,
so that
\[
\mathbf{P}_{3}=\left(\begin{array}{cccc}
1 & \sqrt{p_{2}} & \sqrt{p_{4}} & 0\\
\sqrt{p_{2}} & 1 & \sqrt{p_{3}} & 0\\
\sqrt{p_{4}} & \sqrt{p_{3}} & 1 & \sqrt{1-p_{3}}\\
0 & 0 & \sqrt{1-p_{3}} & 1
\end{array}\right)\approx\left(\begin{array}{cccc}
1 & 0.80 & 0.51 & 0\\
0.80 & 1 & 0.64 & 0\\
0.51 & 0.64 & 1 & 0.77\\
0 & 0 & 0.77 & 1
\end{array}\right).
\]
If $X_{i}\stackrel{\mathrm{iid}}{\sim}\mathrm{Pareto}\left(\alpha,\beta\right)$,
footnote \ref{fn:Pareto} and equation (\ref{eq:2Cov}) give
\[
\frac{\eta}{\sigma}=\left(\sqrt[\beta]{2}-1\right)\sqrt{\beta\left(\beta-2\right)}\longrightarrow\left\{ \begin{array}{rl}
0 & \textrm{as }\beta\rightarrow2^{+}\\
\log2 & \textrm{as }\beta\rightarrow\infty,
\end{array}\right.
\]
so that $\nicefrac{\eta}{\sigma}$ grows monotonically in $\beta$.
Then, for $\beta=2.01$ and $\beta=3$, we have
\[
\mathbf{P}_{3}\approx\left(\begin{array}{cccc}
1 & 0.06 & 0.04 & 0\\
0.06 & 1 & 0.64 & 0\\
0.04 & 0.64 & 1 & 0.77\\
0 & 0 & 0.77 & 1
\end{array}\right)\;\textrm{and }\;\mathbf{P}_{3}\approx\left(\begin{array}{cccc}
1 & 0.45 & 0.29 & 0\\
0.45 & 1 & 0.64 & 0\\
0.29 & 0.64 & 1 & 0.77\\
0 & 0 & 0.77 & 1
\end{array}\right).
\]
$\mathrm{Pareto}\left(\alpha,2.01\right)$'s tendency towards extreme
outliers exceeds that of $\mathrm{Pareto}\left(\alpha,3\right)$;
this tamps down $\mathrm{Cor}\left(\textrm{\textsc{mean}},\textrm{\textsc{median}}\right)$
and $\mathrm{Cor}\left(\textrm{\textsc{mean}},\textrm{\textsc{remedian}}\right)$.

\end{restatable}

\begin{restatable}{myremark}{absvssd}

\label{rem:mallows}As in \citet{M91} we note that 
\[
\sigma\coloneqq\sqrt{\mathrm{Var}\left(X_{1}\right)}\coloneqq\sqrt{\mathbb{E}\left[\left(X_{1}-\mu\right)^{2}\right]}\stackrel{\mathrm{(a)}}{\geq}\mathbb{E}\left|X_{1}-\mu\right|\stackrel{\mathrm{(b)}}{\geq}\mathbb{E}\left|X_{1}-\bar{\mu}\right|\eqqcolon\eta,
\]
where (a) uses Jensen's inequality and (b) uses $\bar{\mu}=\arg\min_{a\in\mathbb{R}}\mathbb{E}\left|X_{1}-a\right|$.

\end{restatable}

While Theorem \ref{thm:rem_med_to_quad_normal} includes three measures
of location---the mean, median, and remedian, Corollary \ref{cor:locDiffs}
studies the differences of pairs of these as $b\rightarrow\infty$.
We ask, \emph{e.g.}, how well does the remedian approximate the sample
median?

\begin{restatable}{corollary}{diff-with-rank}

\label{cor:locDiffs}Let $\iota\coloneqq\nicefrac{1}{f\left(\bar{\mu}\right)}$.
Under (\ref{eq:Xi}) and (\ref{eq:AssumptionCOI_half}) we have
\begin{align}
b^{\nicefrac{k}{2}}\left[X_{\left(R_{k,b}\right)}-X_{\left(h_{k,b}\right)}\right] & \stackrel[\infty]{b}{\Longrightarrow}\mathcal{N}\left(0,\:\nicefrac{\iota^{2}\tau_{k}^{2}}{4}-\nicefrac{\iota^{2}}{4}\right)\label{eq:remMinusMed}\\
b^{\nicefrac{k}{2}}\left[\left(X_{\left(h_{k,b}\right)}-\bar{X}_{b^{k}}\right)-\left(\bar{\mu}-\mu\right)\right] & \stackrel[\infty]{b}{\Longrightarrow}\mathcal{N}\left(0,\:\sigma^{2}-\iota\eta+\nicefrac{\iota^{2}}{4}\right)\label{eq:medMinusMean}\\
b^{\nicefrac{k}{2}}\left[\left(X_{\left(R_{k,b}\right)}-\bar{X}_{b^{k}}\right)-\left(\bar{\mu}-\mu\right)\right] & \stackrel[\infty]{b}{\Longrightarrow}\mathcal{N}\left(0,\:\sigma^{2}-\iota\eta+\nicefrac{\iota^{2}\tau_{k}^{2}}{4}\right).\label{eq:remMinusMean}
\end{align}

\end{restatable}

When $k=1$ (\emph{i.e.}, the remedian is the median), the variance
in (\ref{eq:remMinusMed}) equals zero and the variances in (\ref{eq:medMinusMean})
and (\ref{eq:remMinusMean}) equal each other, as desired (\emph{cf.}\ Theorems
\ref{thm:rem_to_normal} and \ref{thm:rem_rank_to_normal}). By (\ref{eq:remMinusMed}),
the remedian's error in approximating the sample median approaches
zero in probability, as $b\rightarrow\infty$, at rate $C_{k}b^{\nicefrac{-k}{2}}$,
where $C_{k}\coloneqq\nicefrac{\bar{\tau}_{k}}{2f\left(\bar{\mu}\right)}$.
Panels (\ref{eq:medMinusMean}) and (\ref{eq:remMinusMean}) describe
systematic ($\bar{\mu}-\mu$) and stochastic (of order $b^{\nicefrac{-k}{2}}$)
errors in using the median or the remedian to approximate the sample
mean, the former vanishing when $\bar{\mu}=\mu$. Finally, $\tau_{k}^{2}\geq1$
and Remark \ref{rem:mallows} imply that:
\[
\left(\sigma^{2}-\iota\eta+\nicefrac{\iota^{2}\tau_{k}^{2}}{4}\right)\geq\left(\sigma^{2}-\iota\eta+\nicefrac{\iota^{2}}{4}\right)\geq\left(\eta^{2}-\iota\eta+\nicefrac{\iota^{2}}{4}\right)=\left(\eta-\nicefrac{\iota}{2}\right)^{2}\geq0.
\]
That is to say, the variances in (\ref{eq:medMinusMean}) and (\ref{eq:remMinusMean})
make sense (\emph{i.e.}, are non-negative); also, the variance in
(\ref{eq:remMinusMean}) exceeds that in (\ref{eq:medMinusMean})
when $k\geq2$, as expected.

\subsection{Asymptotic Relative Efficiency \label{subsec:Asymptotic-Relative-Efficiency}}

Theorem \ref{thm:rem_med_to_quad_normal} in hand, we now ask, which
of its three, asymptotically-unbiased measures of location most efficiently
uses its $n\coloneqq b^{k}$ data points to approximate the estimand
of interest, the population median $\bar{\mu}$? Statements comparing
the efficiencies of estimators depend on the distribution $F$. The
fact below assumes normality, which implies that $\bar{\mu}=\mu$.

\begin{restatable}{myfact}{meanmedian}

\label{fact:meanMedian}For $0<\sigma_{0}<\infty$ known and $X_{1},X_{2},\ldots,X_{n}\stackrel{\mathrm{iid}}{\sim}\mathcal{N}\left(\bar{\mu},\sigma_{0}^{2}\right)$,
\begin{enumerate}
\item \label{enu:Fisher}The Fisher information in $\bar{X}_{n}\sim\mathcal{N}\left(\bar{\mu},\nicefrac{\sigma_{0}^{2}}{n}\right)$
about $\bar{\mu}$ is $\mathscr{I}\left(\bar{\mu}\right)\coloneqq\nicefrac{n}{\sigma_{0}^{2}}$;
\item \label{enu:ave}The sample mean $\bar{X}_{n}$ efficiently estimates
$\bar{\mu}$; and
\item \label{enu:med}When $n\rightarrow\infty$, the sample median, $\mathrm{median}\left(\mathbf{X}\right)$,
is $100\left(\nicefrac{2}{\pi}\right)\%\approx63.7\%$ as efficient
as $\bar{X}_{n}$ in estimating $\bar{\mu}$.
\end{enumerate}
\end{restatable}
\begin{proof}
Parts \ref{enu:Fisher}, \ref{enu:ave}, and \ref{enu:med} are Examples
2.80, 5.14, and 7.44 on pages 111, 301, and 413 of \citet{S95}. (Example
7.44 has a small typo.) Part \ref{enu:ave} uses part \ref{enu:Fisher}
and the Cram\'er-Rao lower bound (Theorem 5.13 of \citet{S95}).
\end{proof}
While the tails of the normal distribution quickly approach zero,
the sample mean efficiently estimates its center, $\bar{\mu}$. By
the Cram\'er-Rao lower bound, $\nicefrac{1}{\mathscr{I}\left(\bar{\mu}\right)}$
provides a lower bound for the variance of \emph{any} unbiased estimator
of $\bar{\mu}$ (part \ref{enu:ave}). The sample median squanders
$100\left(1-\nicefrac{2}{\pi}\right)\%\approx36.3\%$ of its data,
as $n\rightarrow\infty$; the \emph{asymptotic relative efficiency}
(ARE) of $M_{n}\coloneqq\mathrm{median}\left(\mathbf{X}\right)$ relative
to $\bar{X}_{n}$ is
\[
\textrm{\textsc{are}}\left(M_{n},\bar{X}_{n}\right)\coloneqq\frac{\lim_{n\rightarrow\infty}\mathrm{Var}\left(\sqrt{n}\left(\bar{X}_{n}-\bar{\mu}\right)\right)}{\lim_{n\rightarrow\infty}\mathrm{Var}\left(\sqrt{n}\left(M_{n}-\bar{\mu}\right)\right)}=\frac{\sigma_{0}^{2}}{\left.\pi\sigma_{0}^{2}\right/2}=\nicefrac{2}{\pi}
\]
(see Corollary \ref{cor:bahadurcordist}). Equivalently, $\mathrm{Var}\left(\bar{X}_{n}\right)\sim\mathrm{Var}\left(M_{\nicefrac{\pi n}{2}}\right)$,
as $n\rightarrow\infty$ (part \ref{enu:med}). The ARE measures the
efficiency of one estimator relative to another, under a certain $F$,
as $n\rightarrow\infty$. Efficient estimators quickly zero in on
their estimands.

While Fact \ref{fact:meanMedian} gives $\textrm{\textsc{are}}\left(M_{n},\bar{X}_{n}\right)$
under normality, Propositions \ref{prop:medRemed} and \ref{prop:meanRemed}
give $\textrm{\textsc{are}}\left(X_{\left(R_{k,b}\right)},M_{n}\right)$
and $\textrm{\textsc{are}}\left(X_{\left(R_{k,b}\right)},\bar{X}_{n}\right)$
under any $F$ satisfying (\ref{eq:AssumptionCOI_half}).

\begin{restatable}{proposition}{medremed}

\label{prop:medRemed}Under (\ref{eq:Xi}) and (\ref{eq:AssumptionCOI_half})
the remedian is $100\left(\nicefrac{2}{\pi}\right)^{k-1}\%\leq100\%$
as efficient as the sample median at estimating $\bar{\mu}$ as $b\rightarrow\infty$.

\end{restatable}
\begin{proof}
This follows from Corollary \ref{cor:bahadurcordist} and Theorem
\ref{thm:rem_to_normal} above and Theorem 6.7 on page 473 of \citet{LC98}.
\end{proof}
As expected, the remedian's efficiency falls below that of the sample
median when $k\geq2$ (Theorems \ref{thm:rem_to_normal} and \ref{thm:rem_rank_to_normal}).
The following proposition assumes $\bar{\mu}=\mu$.

\begin{restatable}{proposition}{meanremed}

\label{prop:meanRemed}Under (\ref{eq:Xi}) and (\ref{eq:AssumptionCOI_half})
the remedian is $400\left(\nicefrac{2}{\pi}\right)^{k-1}f\left(\bar{\mu}\right)^{2}\sigma^{2}\%$
as efficient as the sample mean in estimating $\bar{\mu}=\mu$ as
$b\rightarrow\infty$.

\end{restatable}
\begin{proof}
This follows from Theorem \ref{thm:rem_to_normal} above along with
the central limit theorem and Theorem 6.7 on page 473 of \citet{LC98}.
\end{proof}
While Proposition \ref{prop:meanRemed} generalizes part \ref{enu:med}
of Fact \ref{fact:meanMedian},\footnote{It includes both the median and the remedian, and it holds for any
$F$ satisfying (\ref{eq:AssumptionCOI_half}).} it has more moving parts: $X_{\left(R_{k,b}\right)}$ loses efficiency
relative to $\bar{X}_{b^{k}}$ when $k$ increases, $f\left(\bar{\mu}\right)$
decreases, or $\sigma^{2}$ decreases. The following corollary studies
three symmetric distributions.

\begin{restatable}{corollary}{tdist}

\label{cor:propEG}Let $T_{\nu}$ have a $t$-distribution with $\nu>2$
degrees of freedom and fix $0<\alpha,\varsigma<\infty$. Proposition
\ref{prop:meanRemed} implies that:
\begin{enumerate}
\item \label{enu:beta}If $X_{1},X_{2},\ldots,X_{b^{k}}\stackrel{\mathrm{iid}}{\sim}\mathrm{Beta}\left(\alpha,\alpha\right)\left(\stackrel[\mathrm{tiny}]{\alpha}{\approx}\mathrm{Bernoulli}\left(\frac{1}{2}\right),\stackrel[\mathrm{big}]{\alpha}{\approx}\mathcal{N}\left(\frac{1}{2},\frac{1}{4\left(2\alpha+1\right)}\right)\right)$,
the remedian is
\[
\frac{100\left(\nicefrac{2}{\pi}\right)^{k-1}}{16^{\alpha-1}\left(2\alpha+1\right)B\left(\alpha,\alpha\right)^{2}}\%\longrightarrow\left\{ \begin{array}{rl}
0\% & \textrm{ as }\alpha\rightarrow0^{+}\\
100\left(\nicefrac{2}{\pi}\right)^{k}\% & \textrm{ as }\alpha\rightarrow\infty
\end{array}\right.
\]
as efficient as the sample mean in estimating $\bar{\mu}=\mu=\nicefrac{1}{2}$
as $b\rightarrow\infty$.
\item \label{enu:tdist}If $X_{1},X_{2},\ldots,X_{b^{k}}\stackrel{\mathrm{iid}}{\sim}\mathscr{L}\left(\varsigma T_{\nu}+\bar{\mu}\right)\stackrel[\infty]{\nu}{\longrightarrow}\mathcal{N}\left(\bar{\mu},\varsigma^{2}\right)$,
the remedian is 
\[
\frac{400\left(\nicefrac{2}{\pi}\right)^{k-1}}{\left(\nu-2\right)B\left(\nicefrac{\nu}{2},\nicefrac{1}{2}\right)^{2}}\%\longrightarrow\left\{ \begin{array}{rl}
\infty\% & \textrm{ as }\nu\rightarrow2^{+}\\
100\left(\nicefrac{2}{\pi}\right)^{k}\% & \textrm{ as }\nu\rightarrow\infty
\end{array}\right.
\]
as efficient as the sample mean in estimating $\bar{\mu}=\mu$ as
$b\rightarrow\infty$.
\item \label{enu:normal}If $X_{1},X_{2},\ldots,X_{b^{k}}\stackrel{\mathrm{iid}}{\sim}\mathcal{N}\left(\bar{\mu},\sigma^{2}\right)$,
the remedian is $100\left(\nicefrac{2}{\pi}\right)^{k}\%<100\%$ as
efficient as the sample mean in estimating $\bar{\mu}=\mu$ as $b\rightarrow\infty$.
\end{enumerate}
\end{restatable}

Note that Proposition \ref{prop:meanRemed} implies Corollary \ref{cor:propEG}.\ref{enu:normal}
implies Fact \ref{fact:meanMedian}.\ref{enu:med} and that Corollary
\ref{cor:propEG} parts \ref{enu:beta} and \ref{enu:tdist} include
two interesting extremes:
\begin{enumerate}
\item If $F=\mathrm{Beta}\left(\alpha,\alpha\right)$ and $\alpha\rightarrow0^{+}$,
the efficiency of $\bar{X}_{n}$ infinitely exceeds that of $X_{\left(R_{k,b}\right)}$
because $f\left(\bar{\mu}\right)\rightarrow0^{+}$; 
\item If $F=\mathscr{L}\left(\varsigma T_{\nu}+\bar{\mu}\right)$ and $\nu\rightarrow2^{+}$,
the efficiency of $X_{\left(R_{k,b}\right)}$ infinitely exceeds that
of $\bar{X}_{n}$ because $\sigma^{2}\rightarrow\infty$, producing
many outliers (imagine Figure \ref{fig:ThmMultiNormal} with $\mathrm{Pareto}\left(1,2.01\right)$
data).
\end{enumerate}

\section{Asymptotic Normality of General Estimators \label{sec:Other-Quantiles}}

While our paper so far focuses on estimating the population median
$\bar{\mu}$, we now turn to general estimands: $\tilde{\mu}\coloneqq F^{-1}\left(\tilde{p}\right)$,
for $\tilde{p}\in\left(0,1\right)$; see (\ref{eq:AssumptionCOI}).
In particular, we propose an asymptotic distribution for the low-memory
estimator that results when we use the remedian to estimate the $\ell\geq1$
quantiles $\tilde{\mu}_{j}\coloneqq F^{-1}\left(\tilde{p}_{j}\right)$
(see Conjecture \ref{conj:quantRemedian}). We prove the case $\ell=1$.
Section \ref{subsec:Component-wise} gives preliminaries to our approach,
which expands on an idea from \citet{CL93}; see section \ref{subsec:Any-Quantile}. 

\subsection{Medians and Remedians of Vector Components \label{subsec:Component-wise}}

The following corollary of Lemma \ref{lem:bahadur66} gives the limiting
distribution of the vector that holds the sample quantiles of the
components of i.i.d.\ random vectors.

\begin{restatable}{corollary}{compMedian}

\label{cor:compMedian}For $H\left(\mathbf{x}\right)\coloneqq\Pr\left(X_{1}\leq x_{1},\ldots,X_{\ell}\leq x_{\ell}\right)$,
$H_{j}\left(x\right)\coloneqq\Pr\left(X_{j}\leq x\right)$, $\tilde{\mathbf{p}}\in\left(0,1\right)^{\ell}$,
and $\tilde{\mu}_{j}\coloneqq H_{j}^{-1}\left(\tilde{p}_{j}\right)$,
assume (\ref{eq:AssumptionCOI}) for $\tilde{\mu}_{1},\tilde{\mu}_{2},\ldots,\tilde{\mu}_{\ell}$.
For $\mathbf{X}_{1},\ldots,\mathbf{X}_{n}$ i.i.d.\ $H$ and $Y_{n,j}$
the value with rank $\left\lfloor 1+\left(n-1\right)\tilde{p}_{j}\right\rfloor $
in $X_{1,j},X_{2,j},\ldots,X_{n,j}$, 
\begin{equation}
\sqrt{n}\left(\mathbf{Y}_{n}-\tilde{\boldsymbol{\mu}}\right)\stackrel[\infty]{n}{\implies}\mathcal{N}_{\ell}\left(\mathbf{0},\,\mathbf{D\boldsymbol{\Sigma}D}\right),\label{eq:quantCompLim}
\end{equation}
where $\mathbf{D}\coloneqq\mathrm{diag}\left(\left(h_{j}\left(\tilde{\mu}_{j}\right)\right)_{1\leq j\leq\ell}\right)^{-1}$,
$\tilde{\pi}{}_{j,l}\coloneqq\Pr\left(X_{1,j}\leq\tilde{\mu}_{j},\,X_{1,l}\leq\tilde{\mu}_{l}\right)$,
and
\begin{equation}
\boldsymbol{\Sigma}\coloneqq\left(\begin{array}{cccc}
\tilde{p}_{1}\left(1-\tilde{p}_{1}\right) & \tilde{\pi}{}_{1,2}-\tilde{p}_{1}\tilde{p}_{2} & \cdots & \tilde{\pi}{}_{1,\ell}-\tilde{p}_{1}\tilde{p}_{\ell}\\
\tilde{\pi}{}_{1,2}-\tilde{p}_{1}\tilde{p}_{2} & \tilde{p}_{2}\left(1-\tilde{p}_{2}\right) & \cdots & \tilde{\pi}{}_{2,\ell}-\tilde{p}_{2}\tilde{p}_{\ell}\\
\vdots & \vdots & \ddots & \vdots\\
\tilde{\pi}{}_{1,\ell}-\tilde{p}_{1}\tilde{p}_{\ell} & \tilde{\pi}{}_{2,\ell}-\tilde{p}_{2}\tilde{p}_{\ell} & \cdots & \tilde{p}_{\ell}\left(1-\tilde{p}_{\ell}\right)
\end{array}\right).\label{eq:SigmaMat}
\end{equation}

\end{restatable}
\begin{proof}
Use Lemma \ref{lem:bahadur66}, the multivariate CLT, (\ref{eq:AssumptionCOI})
for $\tilde{\mu}_{1},\tilde{\mu}_{2},\ldots,\tilde{\mu}_{\ell}$,
and
\begin{multline}
\Pr\left(X_{1,j}>\tilde{\mu}_{j},\,X_{1,l}>\tilde{\mu}_{l}\right)-\Pr\left(X_{1,j}>\tilde{\mu}_{j}\right)\Pr\left(X_{1,l}>\tilde{\mu}_{l}\right)\\
=\Pr\left(X_{1,j}\leq\tilde{\mu}_{j},\,X_{1,l}\leq\tilde{\mu}_{l}\right)-\Pr\left(X_{1,j}\leq\tilde{\mu}_{j}\right)\Pr\left(X_{1,l}\leq\tilde{\mu}_{l}\right).\label{eq:crossMult}
\end{multline}
The proof of Theorem 2.1 in \citet{BR88} mistakenly uses $n-Z_{n}$
in place of $Z_{n}$ when applying Lemma \ref{lem:bahadur66} and
therefore overlooks the need for (\ref{eq:crossMult}).
\end{proof}
While Corollary \ref{cor:compMedian} applies, \emph{e.g.}, the median
to the components of the $\mathbf{X}_{i}$, we now consider what happens
when we replace the median with the \emph{re}median. The median to
remedian analogy of Corollary \ref{cor:bahadurcordist} to Theorem
\ref{thm:rem_to_normal} suggests that:

\begin{restatable}{conjecture}{compRemedian}

\label{conj:compRemedian}Let $H\left(\mathbf{x}\right)\coloneqq\Pr\left(X_{1}\leq x_{1},\ldots,X_{\ell}\leq x_{\ell}\right)$,
$H_{j}\left(x\right)\coloneqq\Pr\left(X_{j}\leq x\right)$, and $\bar{\mu}_{j}\coloneqq H_{j}^{-1}\left(\nicefrac{1}{2}\right)$.
Assume (\ref{eq:AssumptionCOI}) for $\bar{\mu}_{1},\bar{\mu}_{2},\ldots,\bar{\mu}_{\ell}$.
For $\mathbf{X}_{1},\ldots,\mathbf{X}_{b^{k}}$ i.i.d.\ $H$ and
$Y_{k,b;j}$ the $k\times b$ remedian estimate for values in $X_{1,j},X_{2,j},\ldots,X_{b^{k},j}$,
\[
b^{k/2}\left(\mathbf{Y}_{k,b}-\bar{\boldsymbol{\mu}}\right)\stackrel[\infty]{b}{\implies}\mathcal{N}_{\ell}\left(\mathbf{0},\,\left(\nicefrac{\pi}{2}\right)^{k-1}\mathbf{D\boldsymbol{\Sigma}D}\right),
\]
where $\mathbf{D}$, $\boldsymbol{\Sigma}$ are as in Corollary \ref{cor:compMedian}
with $\bar{p}_{j}=\nicefrac{1}{2}$ and $\bar{\mu}_{j}$ instead of
$\tilde{p}_{j}$ and $\tilde{\mu}_{j}$.

\end{restatable}

That is, Corollary \ref{cor:bahadurcordist} and Theorem \ref{thm:rem_to_normal}
suggest that we multiply the variance on the RHS of (\ref{eq:quantCompLim})
by $\left(\nicefrac{\pi}{2}\right)^{k-1}$ (\emph{cf.}\ Figure \ref{fig:OtherQuant}).
Conjecture \ref{conj:quantRemedian} applies Conjecture \ref{conj:compRemedian}
to remedian estimators for $\tilde{\mu}_{j}\coloneqq F^{-1}\left(\tilde{p}_{j}\right)$
with $0<\tilde{p}_{1}<\tilde{p}_{2}<\cdots<\tilde{p}_{\ell}<1$.

\subsection{Remedian Estimates for Multiple Quantiles \label{subsec:Any-Quantile}}

\citet{CL93} describe the following remedian-based estimator of $F^{-1}\left(p\right)$,
for any $p\in\left(0,1\right)$. Start by finding $1\leq K\leq N$
for which 
\begin{equation}
\tilde{p}\coloneqq B_{K,N}^{-1}\left(\nicefrac{1}{2}\right)\approx p,\label{eq:defptild}
\end{equation}
for $B_{K,N}$ as in (\ref{eq:BjnDef}). Fixing $b\in\mathcal{B}_{3}$
and $k\geq1$, we add a buffer of size $N$ above the usual $k\times b$
remedian. With $n\coloneqq Nb^{k}$, we imagine a streaming setting
in which the data points $X_{1},X_{2},\ldots,X_{n}$ i.i.d.\ $F$
first enter the $N$-buffer. Whenever the $N$-buffer becomes full,
it passes its $K$th order statistic on to the first row of the remedian
and reverts to the empty state. Let $\tilde{\mu}\coloneqq F^{-1}\left(\tilde{p}\right)$
and
\[
X_{\left(K:N;R_{k,b}:b^{k}\right)}\coloneqq\textrm{the value returned after processing all }n\textrm{ data points.}
\]
Then, putting $m_{K,N}\coloneqq\min\left(K,N-K+1\right)$ and citing
\S\ref{subsec:Breakdown-Point} and \S\ref{subsec:Remedian's-Distribution},
we have
\begin{equation}
\epsilon_{K:N;R_{k,b}:b^{k}}^{*}\coloneqq\frac{m_{K,N}}{N}\left(\frac{\left\lceil \nicefrac{b}{2}\right\rceil }{b}\right)^{k},\quad X_{\left(K:N;R_{k,b}:b^{k}\right)}\sim\Psi_{b}^{\left(k\right)}\circ B_{K,N}\circ F\label{eq:bdp-distn}
\end{equation}
for the finite-$n$ breakdown point and distribution of $X_{\left(K:N;R_{k,b}:b^{k}\right)}$.
With data from $G\coloneqq B_{K,N}\circ F$ we obtain the following
corollaries of Theorems \ref{thm:asConv} and \ref{thm:rem_to_normal}.

\begin{restatable}{corollary}{asOtherQuant}

\label{cor:asConvOther}Under (\ref{eq:Xi}) and (\ref{eq:AssumptionCOI})
we have $\Pr\left(\lim_{b\rightarrow\infty}X_{\left(K:N;R_{k,b}:b^{k}\right)}=\tilde{\mu}\right)=1$.

\end{restatable}

\begin{restatable}{corollary}{otherToNormal}

\label{cor:other_to_normal}Under (\ref{eq:Xi}) and (\ref{eq:AssumptionCOI})
and letting $\beta_{K,N}\left(x\right)\coloneqq\frac{d}{dx}B{}_{K,N}\left(x\right)$
we have
\[
2b^{\nicefrac{k}{2}}\beta_{K,N}\left(\tilde{p}\right)f\left(\tilde{\mu}\right)\left(X_{\left(K:N;R_{k,b}:b^{k}\right)}-\tilde{\mu}\right)\stackrel[\infty]{b}{\Longrightarrow}\mathcal{N}\left(0,\left(\nicefrac{\pi}{2}\right)^{k-1}\right).
\]

\end{restatable}

Results like these, with $k,b\rightarrow\infty$ in the latter case,
appear in \citet{CC05}. The following multivariate proposal moves
beyond \citet{CC05}:

\begin{restatable}{conjecture}{quantRemedian}

\label{conj:quantRemedian}Fix $1\leq K_{1}<\cdots<K_{\ell}\leq N$,
and let $X_{1},X_{2},\ldots,X_{Nb^{k}}\stackrel{\mathrm{iid}}{\sim}F$.
For $1\leq i\leq b^{k}$ let $\mathcal{X}_{i}\coloneqq\left\{ X_{\left(i-1\right)N+1},X_{\left(i-1\right)N+2},\ldots,X_{iN}\right\} $.
For $1\leq j\leq\ell$ pass the $K_{j}$th order statistic of each
$\mathcal{X}_{i}$ to the $j$th remedian, and let the $j$th coordinate
of $\mathbf{X}_{\left(\mathbf{K}:N;\mathbf{R}_{k,b}:b^{k}\right)}$
hold the value the $j$th remedian returns. With 
\[
\tilde{\mu}_{j}\coloneqq F^{-1}\left(\tilde{p}_{j}\right)\coloneqq F^{-1}\left(B_{K_{j},N}^{-1}\left(\nicefrac{1}{2}\right)\right),
\]
for $1\leq j\leq\ell$, as in (\ref{eq:defptild}), assume that (\ref{eq:AssumptionCOI})
holds for each $\tilde{\mu}_{1},\tilde{\mu}_{2},\ldots,\tilde{\mu}_{\ell}$.
Then, 
\[
b^{\nicefrac{k}{2}}\left(\mathbf{X}_{\left(\mathbf{K}:N;\mathbf{R}_{k,b}:b^{k}\right)}-\tilde{\boldsymbol{\mu}}\right)\stackrel[\infty]{b}{\Longrightarrow}\mathcal{N}_{\ell}\left(\mathbf{0},\,\left(\nicefrac{\pi}{2}\right)^{k-1}\mathbf{D}\boldsymbol{\Sigma}\mathbf{D}\right),
\]
where
\[
\mathbf{D}\coloneqq\mathrm{diag}\left(\left(\beta_{K_{j}:N}\left(\tilde{p}_{j}\right)f\left(\tilde{\mu}_{j}\right)\right)_{1\leq j\leq\ell}\right)^{-1}
\]
and $\boldsymbol{\Sigma}$ comes from Conjecture \ref{conj:compRemedian}
and has $\bar{\pi}_{j,l}\coloneqq\Pr\left(Y_{1}\leq\tilde{p}_{j\wedge l},Y_{1}+Y_{2}\leq\tilde{p}{}_{j\vee l}\right)$
for trivariate $\mathbf{Y}\sim\mathrm{Dirichlet}\left(K_{j\wedge l},\,K_{j\vee l}-K_{j\wedge l},\,N+1-K_{j\vee l}\right)$.

\end{restatable}
\begin{proof}
Apply Conjecture \ref{conj:compRemedian} noting that the gaps between
the $F\left(X_{\left(K_{j}:N\right)}\right)$ are $\mathrm{Dirichlet}\left(K_{1},K_{2}-K_{1},K_{3}-K_{2},\ldots,K_{\ell}-K_{\ell-1},N+1-K_{\ell}\right)$.
\end{proof}
\begin{figure}
\hfill{}\subfloat{\includegraphics[scale=0.5]{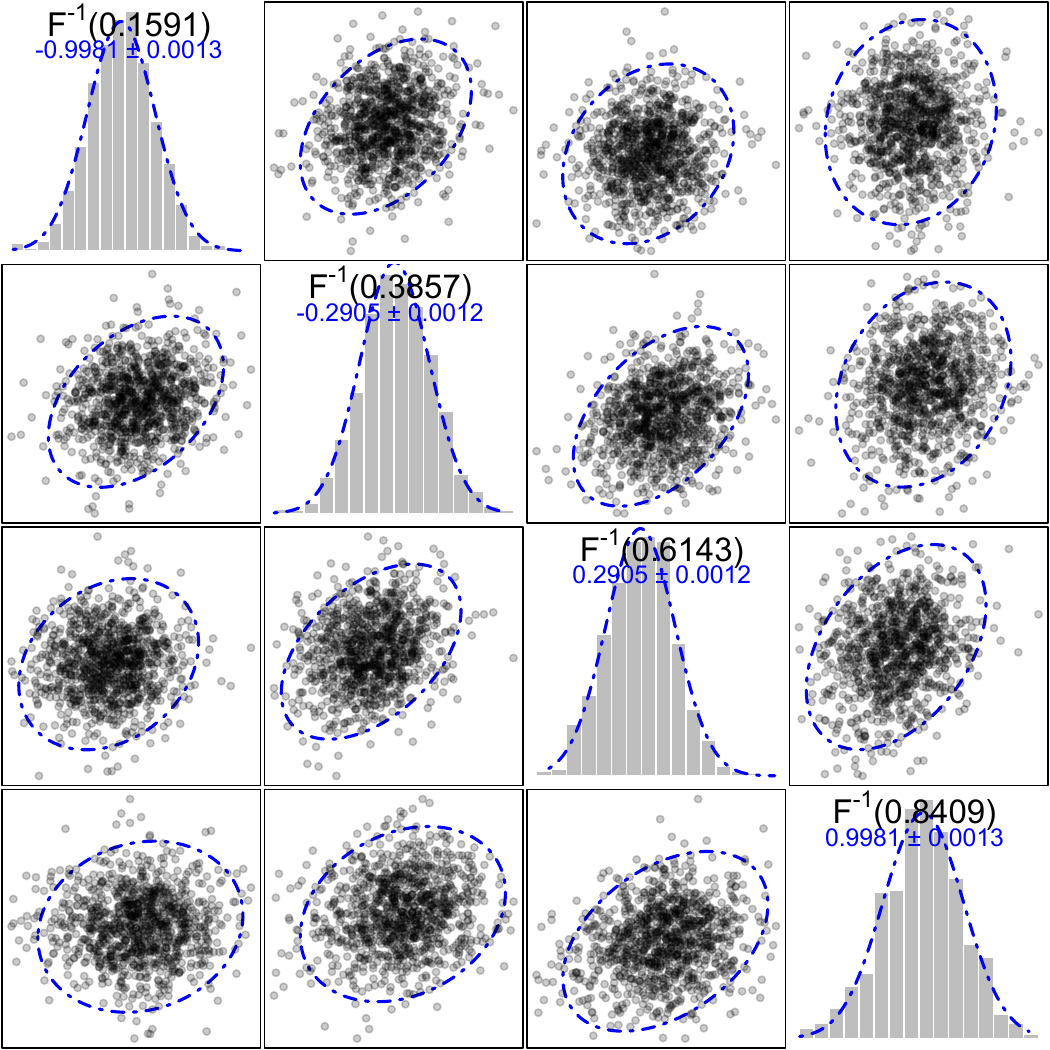}}\hfill{}

\hfill{}\subfloat{\includegraphics[scale=0.5]{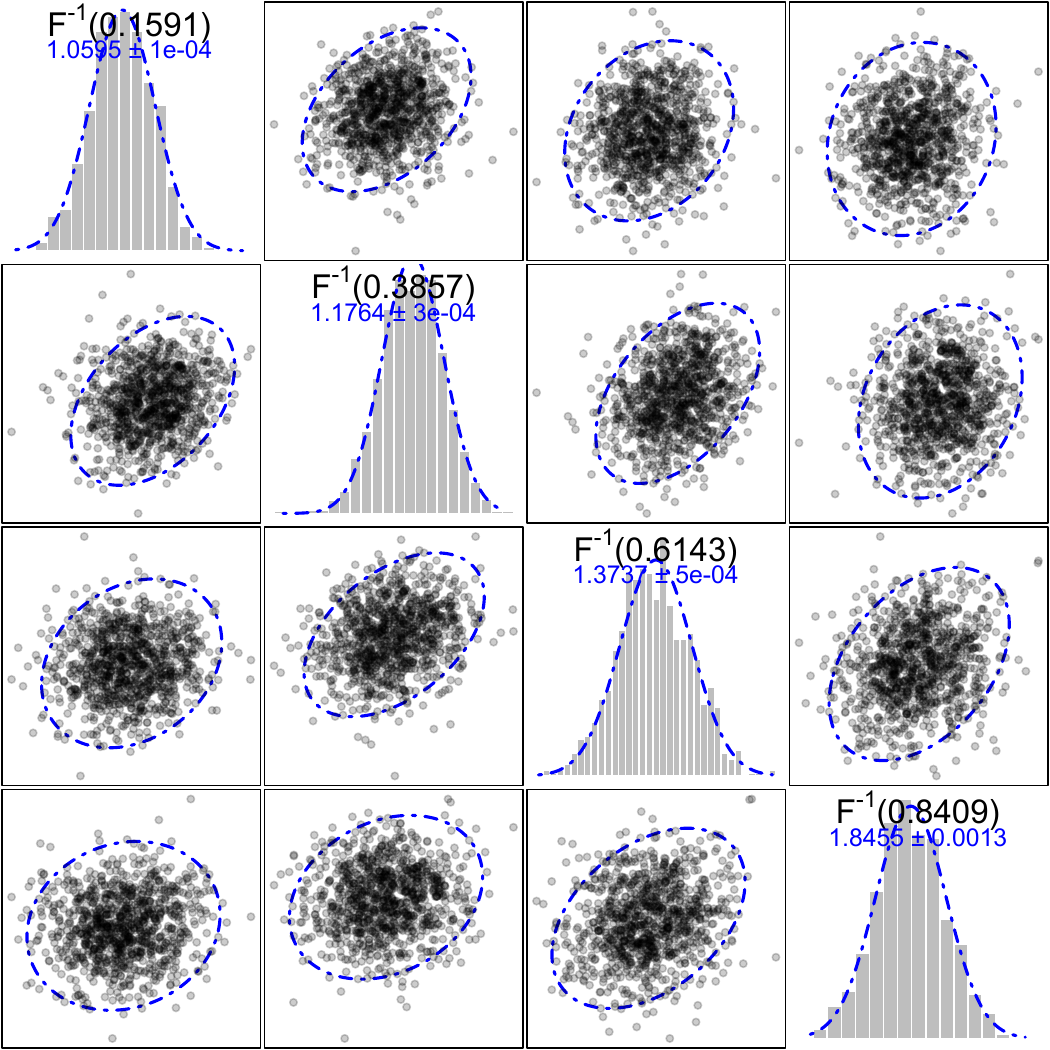}}\hfill{}

\caption{Conjecture \ref{conj:quantRemedian} pairs ($k=3$, $b=101$, $N=4$,
$\mathbf{K}=\left(1,2,3,4\right)$). Figure \ref{fig:ThmMultiNormal}'s
caption describes the curves, the text, and the top and bottom panels.}

\label{fig:OtherQuant}
\end{figure}

While organizations often wish to track $\ell\geq2$ quantiles (\citet{R87,R90,MBLL07,CJLVW06}),
Conjecture \ref{conj:quantRemedian} proposes a distribution for the
large-$b$, remedian-based estimate of $\ell\geq2$ quantiles. While
Conjecture \ref{conj:compRemedian} implies Conjecture \ref{conj:quantRemedian}
and Figure \ref{fig:OtherQuant} simulates quantities from Conjecture
\ref{conj:quantRemedian}, the results in Figure \ref{fig:OtherQuant}
support both conjectures. It is worth noting that an implementation
of Conjecture \ref{conj:quantRemedian} in the streaming setting requires
$N+\ell kb=\mathcal{O}\left(\sqrt[k]{n}\right)$ space, as $b\rightarrow\infty$.
Finally, the $\ell$-remedian of Conjecture \ref{conj:quantRemedian}
has breakdown point $\min_{1\leq j\leq\ell}\epsilon_{K_{j}:N;R_{k,b}:b^{k}}^{*}$
(\emph{cf.}\ (\ref{eq:bdp-distn})).

\section{Conclusions and Discussion \label{sec:Conclusions-and-Discussion}}

Our paper adds to our understanding of the remedian, a low-memory
estimator of $F^{-1}\left(p\right)$. As $b\rightarrow\infty$ (as
$k\rightarrow\infty$) the remedian uses $\mathcal{O}\left(\sqrt[k]{n}\right)$
space ($\mathcal{O}\left(\log n\right)$ space) and 100$\left(1-2^{1-k}\right)$\%
(100\%) of its central inputs (\S\ref{subsec:Breakdown-Point}). With
$b\rightarrow\infty$, we derive the distribution of the standardized
(mean, median, remedian, remedian rank) vector (\S\ref{subsec:Joint-Asymptotic-Normality});
compare the efficiencies of the mean, median, and remedian (\S\ref{subsec:Asymptotic-Relative-Efficiency});
and propose a distribution for the remedian estimate of $\ell\geq2$
population quantiles, proving the case for $\ell=1$ (\S\ref{subsec:Any-Quantile}).
We contextualize the sample median, \emph{i.e.}, the remedian with
$k=1$. As $k$ increases, estimators become more variable, less efficient,
and less robust.

The remedian presents an interesting setting in which to study how
memory requirements trade off against robustness, efficiency, and
precision. Where---one might ask---does the remedian fit into the
panoply of techniques for streaming quantile estimation (Table \ref{tab:literature})?
The remedian is distinguished by its robustness and our understanding
of how it interacts with population and sample quantities (\emph{e.g.},
Theorem \ref{thm:rem_med_to_quad_normal} and Corollary \ref{cor:locDiffs}).
The large-$b$ remedian, with its non-zero breakdown point and relatively
large size, is the Cadillac of quantile estimators: big and roomy,
but expensive. An organization might use the $\ell$-remedian (\S\ref{subsec:Any-Quantile})
to track important, but potentially-corrupted, data streams.

\appendix

\section{Proof of Theorem \ref{thm:rem_to_normal} \label{sec:CC05-Arg}}

\citet{CC05} unwittingly show that the standardized remedian converges
to normality as $b\rightarrow\infty$ and $k$ remains fixed. The
statement they prove sends both $b$ and $k$ to infinity, but their
proof works when just $b\rightarrow\infty$. We present a slightly
modified version of their proof, starting with the following four
facts and three lemmas, which appear in \citet{CC05}. In what follows
let 
\begin{equation}
\theta_{b}\coloneqq\dot{\Psi}_{b}^{\left(1\right)}\left(\nicefrac{1}{2}\right)=\frac{b!}{2^{2m}m!^{2}}\label{eq:def-beta-b}
\end{equation}
be the height of the $\mathrm{Beta}\left(m+1,m+1\right)$ density
at $x=\nicefrac{1}{2}$ for $b\coloneqq2m+1$ and $m\geq1$.
\begin{fact}
\label{fact:fixed-points}$\Psi_{b}^{\left(k\right)}$ has fixed points
at $0,\nicefrac{1}{2},1$, \emph{i.e.}, $\Psi_{b}^{\left(k\right)}\left(x_{0}\right)=x_{0}$,
for $x_{0}=0,\nicefrac{1}{2},1$.
\end{fact}

\begin{proof}
Let $x_{0}\in\left\{ 0,\nicefrac{1}{2},1\right\} $. Note that $\Psi_{b}^{\left(1\right)}\left(x_{0}\right)=x_{0}$.
Then, by induction on $k\geq1$, we have $\Psi_{b}^{\left(k+1\right)}\left(x_{0}\right)=\Psi_{b}^{\left(1\right)}\left(\Psi_{b}^{\left(k\right)}\left(x_{0}\right)\right)=\Psi_{b}^{\left(1\right)}\left(x_{0}\right)=x_{0}$.
\end{proof}
\begin{fact}
\label{fact:monotonic}If $0\leq\alpha<\beta\leq1$, then $0\leq\Psi_{b}^{\left(k\right)}\left(\alpha\right)<\Psi_{b}^{\left(k\right)}\left(\beta\right)\leq1$.
\end{fact}

\begin{proof}
This is true for $\Psi_{b}^{\left(1\right)}\left(x\right)=\Pr\left(\mathrm{Beta}\left(m+1,m+1\right)\leq x\right)$.
Induction and Fact \ref{fact:fixed-points} yield the result.
\end{proof}
\begin{fact}
\label{fact:deriv_1-2}For any $x\in\left[0,1\right]$, 
\begin{align}
\dot{\Psi}_{b}^{\left(k\right)}\left(x\right) & \leq\dot{\Psi}_{b}^{\left(k\right)}\left(\nicefrac{1}{2}\right)=\theta_{b}^{k}\label{eq:D1Rbkinequality}\\
\left|\ddot{\Psi}_{b}^{\left(1\right)}\left(x\right)\right| & \leq6m\theta_{b-2}\left|1-2x\right|.\label{eq:D2Rb1inequality}
\end{align}
\end{fact}

\begin{proof}
First, it is clear that $\dot{\Psi}_{b}^{\left(1\right)}\left(x\right)\leq\dot{\Psi}_{b}^{\left(1\right)}\left(\nicefrac{1}{2}\right)$,
for all $x\in\left[0,1\right]$. Then, by induction on $k\geq2$,
we note that 
\begin{align}
\dot{\Psi}_{b}^{\left(k\right)}\left(x\right) & =\frac{d}{dx}\left[\Psi_{b}^{\left(1\right)}\left(\Psi_{b}^{\left(k-1\right)}\left(x\right)\right)\right]\\
 & =\dot{\Psi}_{b}^{\left(1\right)}\left(\Psi_{b}^{\left(k-1\right)}\left(x\right)\right)\dot{\Psi}_{b}^{\left(k-1\right)}\left(x\right)\\
 & \leq\dot{\Psi}_{b}^{\left(1\right)}\left(\nicefrac{1}{2}\right)\dot{\Psi}_{b}^{\left(k-1\right)}\left(\nicefrac{1}{2}\right)\label{eq:inequality}\\
 & =\dot{\Psi}_{b}^{\left(1\right)}\left(\Psi_{b}^{\left(k-1\right)}\left(\nicefrac{1}{2}\right)\right)\dot{\Psi}_{b}^{\left(k-1\right)}\left(\nicefrac{1}{2}\right)\label{eq:useFact}\\
 & =\dot{\Psi}_{b}^{\left(k\right)}\left(\nicefrac{1}{2}\right),
\end{align}
where (\ref{eq:inequality}) uses the base case and the induction
hypothesis and (\ref{eq:useFact}) uses Fact \ref{fact:fixed-points}.

Induction also shows that $\dot{\Psi}_{b}^{\left(k\right)}\left(\frac{1}{2}\right)=\theta_{b}^{k}$.
Definition (\ref{eq:def-beta-b}) gives the base case. For $k\geq2$
note that 
\begin{align*}
\dot{\Psi}_{b}^{\left(k\right)}\left(\nicefrac{1}{2}\right) & =\frac{d}{dx}\left[\Psi_{b}^{\left(1\right)}\left(\Psi_{b}^{\left(k-1\right)}\left(\nicefrac{1}{2}\right)\right)\right]\\
 & =\dot{\Psi}_{b}^{\left(1\right)}\left(\Psi_{b}^{\left(k-1\right)}\left(\nicefrac{1}{2}\right)\right)\dot{\Psi}_{b}^{\left(k-1\right)}\left(\nicefrac{1}{2}\right)\\
 & =\dot{\Psi}_{b}^{\left(1\right)}\left(\nicefrac{1}{2}\right)\dot{\Psi}_{b}^{\left(k-1\right)}\left(\nicefrac{1}{2}\right)\\
 & =\theta_{b}\theta_{b}^{k-1}=\theta_{b}^{k}.
\end{align*}

Finally, to obtain inequality (\ref{eq:D2Rb1inequality}), we note
that 
\begin{align*}
\left|\ddot{\Psi}_{b}^{\left(1\right)}\left(x\right)\right| & =\frac{b!}{m!^{2}}mx^{m-1}\left(1-x\right)^{m-1}\left|1-2x\right|\\
 & =\frac{b\left(b-1\right)}{m^{2}}m\dot{\Psi}_{b-2}^{\left(1\right)}\left(x\right)\left|1-2x\right|\\
 & \leq6m\theta_{b-2}\left|1-2x\right|,
\end{align*}
where the last inequality uses $b\left(b-1\right)m^{-2}=2\left(2+\nicefrac{1}{m}\right)$
and (\ref{eq:D1Rbkinequality}).
\end{proof}
\begin{fact}
\label{fact:beta_b}$\theta_{b}=\sqrt{\frac{2b}{\pi}}\left(1+o\left(1\right)\right)$
as $b\rightarrow\infty$.
\end{fact}

\begin{proof}
Using Stirling's approximation and $\lim_{m\rightarrow\infty}\left(\frac{2m+1}{2m}\right)^{2m}=e$,
we have
\[
\frac{\left(2m+1\right)!}{m!^{2}2^{2m}}\sim\sqrt{\frac{2m+1}{2\pi}}\left(\frac{2m+1}{m}\right)\left(\frac{1}{e}\right)\left(\frac{2m+1}{2m}\right)^{2m}\sim\sqrt{\frac{2\left(2m+1\right)}{\pi}}.
\]
\end{proof}
\begin{lem}
\label{lem:bound2ndDeriv}For some generic $C>0$ and $\tau$ such
that $\left|\tau-\nicefrac{1}{2}\right|=\mathcal{O}\left(\frac{1}{\sqrt{b}\theta_{b}^{k-1}}\right)$,
\[
\left|\ddot{\Psi}_{b}^{\left(k-1\right)}\left(\tau\right)\right|\leq C\theta_{b}^{2k-3}.
\]
\end{lem}

\begin{proof}
We proceed by induction, proving two base cases. 
\begin{description}
\item [{Case\ \textit{k}\,=\,1}] Expression (\ref{eq:recursive}) gives
$\ddot{\Psi}_{b}^{\left(0\right)}\left(x\right)\equiv0$, so that
$\left|\ddot{\Psi}_{b}^{\left(0\right)}\left(\tau\right)\right|\leq C\theta_{b}^{-1}$. 
\item [{Case\ \textit{k}\,=\,2}] Facts \ref{fact:deriv_1-2}--\ref{fact:beta_b}
give $\left|\ddot{\Psi}_{b}^{\left(1\right)}\left(\tau\right)\right|\leq6m\theta_{b-2}\left|1-2\tau\right|\leq\frac{Cm\theta_{b-2}}{\sqrt{b}\theta_{b}}\leq C\theta_{b}$.
\item [{Induction\ on\ \textit{k}\,\ensuremath{\ge}\,2}] With $\tau^{*}$
between $\nicefrac{1}{2}$ and $\tau$, Taylor expansion gives
\begin{align}
\Psi_{b}^{\left(k-2\right)}\left(\tau\right) & =\Psi_{b}^{\left(k-2\right)}\left(\nicefrac{1}{2}\right)+\dot{\Psi}_{b}^{\left(k-2\right)}\left(\tau^{*}\right)\left(\tau-\nicefrac{1}{2}\right)\\
 & =\nicefrac{1}{2}+\mathcal{O}\left(\nicefrac{1}{b}\right),\label{eq:taylor_Rb(k-2)}
\end{align}
where (\ref{eq:taylor_Rb(k-2)}) follows from Facts \ref{fact:fixed-points}
and \ref{fact:deriv_1-2}. Then $\left|\ddot{\Psi}_{b}^{\left(k-1\right)}\left(\tau\right)\right|$
equals
\begin{align}
 & \left|\ddot{\Psi}_{b}^{\left(1\right)}\left(\Psi_{b}^{\left(k-2\right)}\left(\tau\right)\right)\left[\dot{\Psi}_{b}^{\left(k-2\right)}\left(\tau\right)\right]^{2}+\dot{\Psi}_{b}^{\left(1\right)}\left(\Psi_{b}^{\left(k-2\right)}\left(\tau\right)\right)\ddot{\Psi}_{b}^{\left(k-2\right)}\left(\tau\right)\right|\label{eq:big_ineq_1}\\
 & \qquad\leq6m\theta_{b-2}\theta_{b}^{2k-4}\left|1-2\Psi_{b}^{\left(k-2\right)}\left(\tau\right)\right|+\theta_{b}\left|\ddot{\Psi}_{b}^{\left(k-2\right)}\left(\tau\right)\right|\label{eq:big_ineq_2}\\
 & \qquad\leq C\theta_{b}^{2k-3}+\theta_{b}\left|\ddot{\Psi}_{b}^{\left(k-2\right)}\left(\tau\right)\right|\label{eq:big_ineq_3}\\
 & \qquad\leq C\theta_{b}^{2k-3}+C\theta_{b}^{2k-4}\leq C\theta_{b}^{2k-3},\label{eq:big_ineq_4}
\end{align}
where (\ref{eq:big_ineq_2}) uses the triangle inequality and Fact
\ref{fact:deriv_1-2}, (\ref{eq:big_ineq_3}) uses expression (\ref{eq:taylor_Rb(k-2)})
and Fact \ref{fact:beta_b}, and (\ref{eq:big_ineq_4}) uses the induction
hypothesis and the generic definition of $C$.
\end{description}
\end{proof}
\begin{lem}
\label{lem:PsiRem}Let $U_{\left(i:k\right)}$ be the $i$th order
statistic of $U_{1},\ldots,U_{k}\stackrel{\mathrm{iid}}{\sim}\mathrm{Uniform}\left(0,1\right)$.
For $m,k\geq1$ and $b\coloneqq2m+1$, we have $\Psi_{b}^{\left(k-1\right)}\left(U_{\left(R_{k,b}:b^{k}\right)}\right)\stackrel{\mathscr{L}}{=}U_{\left(m+1:2m+1\right)}$.\footnote{The argument in \citet{CC05} uses, but does not prove, this statement.}
\end{lem}

\begin{proof}
The remedian's definition and (\ref{eq:recursive}) imply the base
case when $k=1$. Then, if the result holds for $k\geq1$, we have
$\Psi_{b}^{\left(k\right)}\left(U_{\left(R_{k+1,b}:b^{k+1}\right)}\right)$
\begin{align}
 & =\Psi_{b}^{\left(1\right)}\left(\Psi_{b}^{\left(k-1\right)}\left(U_{\left(R_{k+1,b}:b^{k+1}\right)}\right)\right)\label{eq:PsiRem-1}\\
 & \stackrel{\mathscr{L}}{=}\Psi_{b}^{\left(1\right)}\left(\Psi_{b}^{\left(k-1\right)}\left(\mathrm{med}\left(U_{\left(R_{k,b}:b^{k}\right)}^{\left(1\right)},\ldots,U_{\left(R_{k,b}:b^{k}\right)}^{\left(b\right)}\right)\right)\right)\label{eq:PsiRem-2}\\
 & =\Psi_{b}^{\left(1\right)}\left(\mathrm{med}\left(\Psi_{b}^{\left(k-1\right)}\left(U_{\left(R_{k,b}:b^{k}\right)}^{\left(1\right)}\right),\ldots,\Psi_{b}^{\left(k-1\right)}\left(U_{\left(R_{k,b}:b^{k}\right)}^{\left(b\right)}\right)\right)\right)\label{eq:PsiRem-3}\\
 & \stackrel{\mathscr{L}}{=}\Psi_{b}^{\left(1\right)}\left(\mathrm{med}\left(U_{\left(m+1:2m+1\right)}^{\left(1\right)},\ldots,U_{\left(m+1:2m+1\right)}^{\left(b\right)}\right)\right)\label{eq:PsiRem-4}\\
 & =\mathrm{med}\left(\Psi_{b}^{\left(1\right)}\left(U_{\left(m+1:2m+1\right)}^{\left(1\right)}\right),\ldots,\Psi_{b}^{\left(1\right)}\left(U_{\left(m+1:2m+1\right)}^{\left(b\right)}\right)\right)\label{eq:PsiRem-5}\\
 & \stackrel{\mathscr{L}}{=}\mathrm{med}\left(U_{1},\ldots,U_{b}\right),\label{eq:PsiRem-6}
\end{align}
where (\ref{eq:PsiRem-2}) uses the remedian's definition, $\mathrm{med}$
for the median operator, and 
\[
U_{\left(R_{k,b}:b^{k}\right)}^{\left(1\right)},\ldots,U_{\left(R_{k,b}:b^{k}\right)}^{\left(b\right)}\stackrel{\mathrm{iid}}{\sim}\Psi_{b}^{\left(k\right)};
\]
(\ref{eq:PsiRem-3}) and (\ref{eq:PsiRem-5}) use Fact \ref{fact:monotonic};
(\ref{eq:PsiRem-4}) uses the induction hypothesis and 
\[
U_{\left(m+1:2m+1\right)}^{\left(1\right)},\ldots,U_{\left(m+1:2m+1\right)}^{\left(b\right)}\stackrel{\mathrm{iid}}{\sim}\mathrm{Beta}\left(m+1,m+1\right)\equiv\Psi_{b}^{\left(1\right)};
\]
and (\ref{eq:PsiRem-6}) uses $X\sim F\textrm{ and }X\textrm{ continuous}\implies F\left(X\right)\sim\mathrm{Uniform}\left(0,1\right)$.
\end{proof}
\begin{lem}
\label{lem:medianOfUniform}For $m\geq1$ and $U_{1},U_{2},\ldots,U_{2m+1}\stackrel{\mathrm{iid}}{\sim}\mathrm{Uniform}\left(0,1\right)$,
we have
\[
\sup_{x\in\mathbb{R}}\left|\Pr\left(\sqrt{8\left(m+1\right)}\left(U_{\left(m+1\right)}-\nicefrac{1}{2}\right)\leq x\right)-\Phi\left(x\right)\right|\leq C\sqrt{\frac{2m+1}{m\left(m+1\right)}},
\]
for $U_{\left(r\right)}$ the $r$th order statistic, $\Phi\left(x\right)$
the standard normal CDF, and $C>0$.
\end{lem}

\begin{proof}
See Theorem 1 of \citet{E80} or Theorem 4.2.1 of \citet{R89}.
\end{proof}
\remToNormal*
\begin{proof}
By Fact \ref{fact:beta_b}, showing that $\lambda_{b}\left(X_{\left(R_{k,b}\right)}-\bar{\mu}\right)\stackrel[\infty]{b}{\Longrightarrow}\mathcal{N}\left(0,1\right)$
suffices, where 
\begin{equation}
\lambda_{b}\coloneqq\sqrt{8\left(m+1\right)}\theta_{b}^{k-1}f\left(\bar{\mu}\right)\stackrel[\infty]{b}{\sim}2b^{\nicefrac{k}{2}}f\left(\bar{\mu}\right)\left(\nicefrac{2}{\pi}\right)^{\frac{k-1}{2}}.\label{eq:def-lambda-b}
\end{equation}
We pursue the following in order to show that $\lambda_{b}\left(X_{\left(R_{k,b}\right)}-\bar{\mu}\right)\stackrel[\infty]{b}{\Longrightarrow}\mathcal{N}\left(0,1\right)$:
\begin{enumerate}
\item We derive three Taylor series expansions and a limit using the last
two;
\item From Lemma \ref{lem:medianOfUniform}, the limit, and the first Taylor
series we derive the result.
\end{enumerate}
We first derive three Taylor series expansions:
\begin{enumerate}
\item We expand $G_{k-1,b}\left(\nicefrac{x}{\lambda_{b}}+\bar{\mu}\right)$
about $\bar{\mu}$:
\begin{align}
G_{k-1,b}\left(\nicefrac{x}{\lambda_{b}}+\bar{\mu}\right) & =G_{k-1,b}\left(\bar{\mu}\right)+\nicefrac{x}{\lambda_{b}}\dot{G}_{k-1,b}\left(x_{1,b}\right)\label{eq:expand_G}\\
 & =\Psi_{b}^{\left(k-1\right)}\left(F\left(\bar{\mu}\right)\right)+\nicefrac{x}{\lambda_{b}}\dot{\Psi}_{b}^{\left(k-1\right)}\left(F\left(x_{1,b}\right)\right)f\left(x_{1,b}\right)\\
 & =\Psi_{b}^{\left(k-1\right)}\left(\nicefrac{1}{2}\right)+\nicefrac{x}{\lambda_{b}}\dot{\Psi}_{b}^{\left(k-1\right)}\left(F\left(x_{1,b}\right)\right)f\left(x_{1,b}\right)\\
 & =\nicefrac{1}{2}+\nicefrac{x}{\lambda_{b}}\dot{\Psi}_{b}^{\left(k-1\right)}\left(F\left(x_{1,b}\right)\right)f\left(x_{1,b}\right),\label{eq:usefixed}
\end{align}
where $x_{1,b}$ is between $\bar{\mu}$ and $\bar{\mu}+\nicefrac{x}{\lambda_{b}}$
and expression (\ref{eq:usefixed}) uses Fact \ref{fact:fixed-points}. 
\item We expand $F\left(x_{1,b}\right)$ about $\bar{\mu}$:
\begin{equation}
F\left(x_{1,b}\right)=F\left(\bar{\mu}\right)+\nicefrac{x}{\lambda_{b}}f\left(x_{2,b}\right)=\nicefrac{1}{2}+\nicefrac{x}{\lambda_{b}}f\left(x_{2,b}\right),\label{eq:expand_F}
\end{equation}
where $x_{2,b}$ is between $\bar{\mu}$ and $x_{1,b}$.
\item We expand $\dot{\Psi}_{b}^{\left(k-1\right)}\left(F\left(x_{1,b}\right)\right)$
about $\nicefrac{1}{2}$:
\begin{align}
\dot{\Psi}_{b}^{\left(k-1\right)}\left(F\left(x_{1,b}\right)\right) & =\dot{\Psi}_{b}^{\left(k-1\right)}\left(\nicefrac{1}{2}\right)+\nicefrac{x}{\lambda_{b}}f\left(x_{2,b}\right)\ddot{\Psi}_{b}^{\left(k-1\right)}\left(p_{b}\right)\label{eq:expand_Rdot_1}\\
 & =\theta_{b}^{k-1}+\nicefrac{x}{\lambda_{b}}f\left(x_{2,b}\right)\ddot{\Psi}_{b}^{\left(k-1\right)}\left(p_{b}\right),\label{eq:expand_Rdot_2}
\end{align}
where $p_{b}$ is between $\nicefrac{1}{2}$ and $F\left(x_{1,b}\right)$,
(\ref{eq:expand_Rdot_1}) uses expression (\ref{eq:expand_F}), and
(\ref{eq:expand_Rdot_2}) uses Fact \ref{fact:deriv_1-2}.
\end{enumerate}
We now derive the aforementioned limit. Using (\ref{eq:expand_Rdot_2})
and (\ref{eq:def-lambda-b}) we have
\begin{align}
\frac{\dot{\Psi}_{b}^{\left(k-1\right)}\left(F\left(x_{1,b}\right)\right)}{\theta_{b}^{k-1}} & =1+\frac{x}{\sqrt{8\left(m+1\right)}\theta_{b}^{2k-2}}\frac{f\left(x_{2,b}\right)}{f\left(\bar{\mu}\right)}\ddot{\Psi}_{b}^{\left(k-1\right)}\left(p_{b}\right)\\
 & =1+\frac{xf\left(x_{2,b}\right)}{f\left(\bar{\mu}\right)\mathcal{O}\left(\nicefrac{1}{b}\right)}\stackrel[\infty]{b}{\longrightarrow}1,\label{eq:bound_Rdotdot}
\end{align}
where (\ref{eq:bound_Rdotdot}) uses Lemma \ref{lem:bound2ndDeriv}
(which uses (\ref{eq:expand_F}) and (\ref{eq:def-lambda-b}) to show
that $\left|p_{b}-\nicefrac{1}{2}\right|=\mathcal{O}\left(\nicefrac{1}{\lambda_{b}}\right)=\mathcal{O}\left(\nicefrac{1}{\sqrt{b}\theta_{b}^{k-1}}\right)$)
and (\ref{eq:AssumptionCOI}).

Results (\ref{eq:expand_G}) through (\ref{eq:bound_Rdotdot}) in
place, we use Lemmas \ref{lem:PsiRem} and \ref{lem:medianOfUniform}
to complete the proof. Using Lemma \ref{lem:PsiRem}, we have 
\[
G_{k-1,b}\left(X_{\left(R_{k,b}\right)}\right)=\Psi_{b}^{\left(k-1\right)}\left(U_{\left(R_{k,b}\right)}\right)\stackrel{\mathscr{L}}{=}U_{\left(m+1:2m+1\right)},
\]
where Table \ref{tab:notation} defines $A\stackrel{\mathscr{L}}{=}B$.
This means that $\Pr\left(\lambda_{b}\left(X_{\left(R_{k,b}\right)}-\bar{\mu}\right)\leq x\right)$
\begin{align}
 & =\Pr\left(\lambda_{b}\left(G_{k-1,b}^{-1}\left(U_{\left(m+1:2m+1\right)}\right)-\bar{\mu}\right)\leq x\right)\\
 & =\Pr\left(U_{\left(m+1:2m+1\right)}\leq G_{k-1,b}\left(\nicefrac{x}{\lambda_{b}}+\bar{\mu}\right)\right)\\
 & =\Phi\left(\sqrt{8\left(m+1\right)}\left(G_{k-1,b}\left(\nicefrac{x}{\lambda_{b}}+\bar{\mu}\right)-\nicefrac{1}{2}\right)\right)+{\cal O}\left(\nicefrac{1}{\sqrt{b}}\right)\label{eq:useLemma}\\
 & =\Phi\left(\frac{\dot{\Psi}_{b}^{\left(k-1\right)}\left(F\left(x_{1,b}\right)\right)}{\theta_{b}^{k-1}}\frac{f\left(x_{1,b}\right)}{f\left(\bar{\mu}\right)}x\right)+{\cal O}\left(\nicefrac{1}{\sqrt{b}}\right)\stackrel[\infty]{b}{\longrightarrow}\Phi\left(x\right),\label{eq:useDeriv}
\end{align}
where (\ref{eq:useLemma}) uses Lemma \ref{lem:medianOfUniform} and
(\ref{eq:useDeriv}) uses (in order of appearance) (\ref{eq:usefixed}),
(\ref{eq:def-lambda-b}), (\ref{eq:bound_Rdotdot}), and (\ref{eq:AssumptionCOI}).
This completes the proof.
\end{proof}

\section{Proof of Lemma \ref{lem:iter_rem_to_normal} \label{sec:Iterated_Remedians}}

We now prove our iterated remedian result, starting with the following
fact.
\begin{fact}
\label{fact:continuous_R}$\Psi_{b}^{\left(k\right)}:\left[0,1\right]\rightarrow\left[0,1\right]$
is continuous.
\end{fact}

\begin{proof}
$\Psi_{b}^{\left(1\right)}$, the integral of a continuous function,
is continuous. $\Psi_{b}^{\left(k\right)}$ is continuous by induction:
the composition of two continuous functions is continuous.
\end{proof}
\iterRemToNormal*
\begin{proof}
Fact \ref{fact:continuous_R} and assumption (\ref{eq:AssumptionCOI_half})
imply that $\Psi_{b_{t,m}}^{\left(k_{m}\right)}\circ\Psi_{b_{t,m-1}}^{\left(k_{m-1}\right)}\circ\cdots\circ\Psi_{b_{t,1}}^{\left(k_{1}\right)}\circ F$
is continuous on an open interval ${\cal I}$ containing $\bar{\mu}$.
Noting that 
\[
\bar{\mu}+\frac{\left(\nicefrac{\pi}{2}\right)^{\frac{\sum_{i=1}^{m}k_{i}-1}{2}}y}{2f\left(\bar{\mu}\right)\prod_{i=1}^{m}b_{t,i}^{\nicefrac{k_{i}}{2}}}\in\mathcal{I},
\]
for $t$ large enough, we see that the multivariate limit in (\ref{eq:iterLim})
exists. We proceed by induction on $m\geq1$. Theorem \ref{thm:rem_to_normal}
gives the base case $m=1$. Assuming that
\begin{equation}
\Psi_{b_{t,m-1}}^{\left(k_{m-1}\right)}\circ\Psi_{b_{t,m-2}}^{\left(k_{m-2}\right)}\circ\cdots\circ\Psi_{b_{t,1}}^{\left(k_{1}\right)}\circ F\left(\bar{\mu}+\frac{\left(\nicefrac{\pi}{2}\right)^{\frac{\sum_{i=1}^{m-1}k_{i}-1}{2}}y}{2f\left(\bar{\mu}\right)\prod_{i=1}^{m-1}b_{t,i}^{\nicefrac{k_{i}}{2}}}\right)\stackrel{t\rightarrow\infty}{\longrightarrow}\Phi\left(y\right),\label{eq:iter_ind}
\end{equation}
for $m\geq2$, we have
\begin{align}
 & \Psi_{b_{m}}^{\left(k_{m}\right)}\circ\Psi_{b_{t,m-1}}^{\left(k_{m-1}\right)}\circ\cdots\circ\Psi_{b_{t,1}}^{\left(k_{1}\right)}\circ F\left(\bar{\mu}+\frac{\left(\nicefrac{\pi}{2}\right)^{\frac{\sum_{i=1}^{m}k_{i}-1}{2}}y}{2f\left(\bar{\mu}\right)b_{m}^{\nicefrac{k_{m}}{2}}\prod_{i=1}^{m-1}b_{t,i}^{\nicefrac{k_{i}}{2}}}\right)\\
 & \qquad\stackrel{t\rightarrow\infty}{\longrightarrow}\Psi_{b_{m}}^{\left(k_{m}\right)}\circ\Phi\left(\left(\nicefrac{\pi}{2b_{m}}\right)^{\nicefrac{k_{m}}{2}}y\right)=\Psi_{b_{m}}^{\left(k_{m}\right)}\circ\Phi\left(0+\frac{\left(\nicefrac{\pi}{2}\right)^{\frac{k_{m}-1}{2}}y}{2b_{m}^{\nicefrac{k_{m}}{2}}\phi\left(0\right)}\right)\label{eq:first_limit}\\
 & \qquad\stackrel{b_{m}\rightarrow\infty}{\longrightarrow}\Phi\left(y\right),\label{eq:second_limit}
\end{align}
where the first expression in (\ref{eq:first_limit}) uses Fact \ref{fact:continuous_R}
and (\ref{eq:iter_ind}), and (\ref{eq:second_limit}) uses Theorem
\ref{thm:rem_to_normal}. This completes the proof. 
\end{proof}

\section{Proof of Theorem \ref{thm:rem_rank_to_normal} \label{sec:Rank_to_Normal}}

We turn the the asymptotic normality of the remedian rank $R_{k,b}$,
starting with three lemmas. The latter two are well-known and proved
here for completeness.
\begin{lem}
\label{lem:existLim}For $k\geq2$ and $b\in\mathcal{B}_{3}$ let
\begin{align*}
\mathcal{R}_{k,b} & \coloneqq\left\{ \left\lceil \nicefrac{b}{2}\right\rceil ^{k},\left\lceil \nicefrac{b}{2}\right\rceil ^{k}+1,\ldots,b^{k}+1-\left\lceil \nicefrac{b}{2}\right\rceil ^{k}\right\} \\
\overline{\mathcal{R}}_{k,b} & \coloneqq2b^{\nicefrac{k}{2}}\left(b^{-k}\mathcal{R}_{k,b}-\nicefrac{1}{2}\right).
\end{align*}
For every $r\in\mathbb{R}$ there is a sequence of values $r_{j}\in\overline{\mathcal{R}}_{k,b_{j}}$
such that $b_{j}=b_{0}+2j$, for $b_{0}\in\mathcal{B}_{3}$, and $\lim_{j\rightarrow\infty}r_{j}=r$.
\end{lem}

\begin{proof}
Consecutive elements in $\mathcal{R}_{k,b}$ are one unit apart, which
implies that consecutive elements in $\overline{\mathcal{R}}_{k,b}$
are $2b^{\nicefrac{-k}{2}}$ units apart. Further, we note that 
\[
\overline{\mathcal{R}}_{k,b}\stackrel{b\rightarrow\infty}{\sim}\left(1-\left(\nicefrac{1}{2}\right)^{k-1}\right)b^{\nicefrac{k}{2}}\left[-1,1\right],
\]
so that, for $b$ large enough and $\epsilon>0$, $\overline{\mathcal{R}}_{k,b}$
contains elements both larger than $r+\epsilon$ and smaller than
$r-\epsilon$. This completes the proof.
\end{proof}
\begin{lem}
\label{lem:DirK2}If $B_{n}\sim\mathrm{Beta}\left(i_{n},n-i_{n}+1\right)$,
$i_{n}\in\left[n\right]$, and $\frac{i_{n}}{n}\stackrel[\infty]{n}{\longrightarrow}\alpha\in\left(0,1\right)$,
\[
\sqrt{n}\left(B_{n}-\frac{i_{n}}{n+1}\right)\stackrel[\infty]{n}{\Longrightarrow}\mathcal{N}\left(0,\alpha\left(1-\alpha\right)\right).
\]
\end{lem}

\begin{proof}
Let $\bar{i}_{n}\coloneqq n-i_{n}+1$. Next, for $\xi_{1},\xi_{2},\ldots,\xi_{n+1}\stackrel{\mathrm{iid}}{\sim}\mathrm{Exp}\left(1\right)$,
let 
\begin{align*}
G & \coloneqq\sum_{i=1}^{i_{n}}\xi_{i}\sim\mathrm{Gamma}\left(i_{n},1\right)\\
G' & \coloneqq\sum_{i=i_{n}+1}^{n+1}\xi_{i}\sim\mathrm{Gamma}\left(n-i_{n}+1,1\right)
\end{align*}
so that $G$ and $G'$ are independent, and 
\[
B_{n}\stackrel{\mathscr{L}}{=}\frac{G}{G+G'}\sim\mathrm{Beta}\left(i_{n},n-i_{n}+1\right)
\]
Then,
\begin{align}
\sqrt{n}\left(B_{n}-\frac{i_{n}}{n+1}\right) & \stackrel{\mathscr{L}}{=}\sqrt{n}\left(\frac{\sum_{i=1}^{i_{n}}\xi_{i}-\frac{i_{n}}{n+1}\sum_{i=1}^{n+1}\xi_{i}}{\sum_{i=1}^{n+1}\xi_{i}}\right)\\
 & =\frac{\frac{\sqrt{n}}{n+1}\frac{\bar{i}_{n}}{n+1}\sum_{i=1}^{i_{n}}\xi_{i}-\frac{\sqrt{n}}{n+1}\frac{i_{n}}{n+1}\sum_{i=i_{n}+1}^{n+1}\xi_{i}}{\frac{1}{n+1}\sum_{i=1}^{n+1}\xi_{i}}\label{eq:full-expr}
\end{align}
The first term in the numerator of (\ref{eq:full-expr}) equals
\begin{align}
\frac{\sqrt{n}}{n+1}\frac{\bar{i}_{n}}{n+1}\sum_{i=1}^{i_{n}}\xi_{i} & =\frac{\sqrt{n}}{n+1}\frac{\bar{i}_{n}}{n+1}\left[\sum_{i=1}^{i_{n}}\left(\xi_{i}-1\right)+i_{n}\right]\\
 & =\frac{\sqrt{ni_{n}}}{n+1}\frac{\bar{i}_{n}}{n+1}\frac{1}{\sqrt{i_{n}}}\sum_{i=1}^{i_{n}}\left(\xi_{i}-1\right)+\frac{\sqrt{n}i_{n}\bar{i}_{n}}{\left(n+1\right)^{2}}\label{eq:1st-term}
\end{align}
The second term in the numerator of (\ref{eq:full-expr}) equals 
\begin{align}
-\frac{\sqrt{n}}{n+1}\frac{i_{n}}{n+1}\sum_{i=i_{n}+1}^{n+1}\xi_{i} & =-\frac{\sqrt{n}}{n+1}\frac{i_{n}}{n+1}\left[\sum_{i=i_{n}+1}^{n+1}\left(\xi_{i}-1\right)+\bar{i}_{n}\right]\\
 & =-\frac{\sqrt{n\bar{i}_{n}}}{n+1}\frac{i_{n}}{n+1}\frac{1}{\sqrt{\bar{i}_{n}}}\sum_{i=i_{n}+1}^{n+1}\left(\xi_{i}-1\right)-\frac{\sqrt{n}i_{n}\bar{i}_{n}}{\left(n+1\right)^{2}}\label{eq:2nd-term}
\end{align}
Summing (\ref{eq:1st-term}) and (\ref{eq:2nd-term}) we see that
$\sqrt{n}\left(B_{n}-\frac{i_{n}}{n+1}\right)$ equals
\begin{equation}
\frac{\frac{\sqrt{ni_{n}}}{n+1}\frac{\bar{i}_{n}}{n+1}\frac{1}{\sqrt{i_{n}}}\sum_{i=1}^{i_{n}}\left(\xi_{i}-1\right)-\frac{\sqrt{n\bar{i}_{n}}}{n+1}\frac{i_{n}}{n+1}\frac{1}{\sqrt{\bar{i}_{n}}}\sum_{i=i_{n}+1}^{n+1}\left(\xi_{i}-1\right)}{\frac{1}{n+1}\sum_{i=1}^{n+1}\xi_{i}}\label{eq:combined}
\end{equation}
To see that (\ref{eq:combined}) converges to $\mathcal{N}\left(0,\alpha\left(1-\alpha\right)\right)$,
note that
\begin{enumerate}
\item The denominator converges almost surely to one (SLLN);
\item $\frac{\sqrt{ni_{n}}}{n+1}\frac{\bar{i}_{n}}{n+1}\frac{1}{\sqrt{i_{n}}}\sum_{i=1}^{i_{n}}\left(\xi_{i}-1\right)\stackrel[\infty]{n}{\Longrightarrow}\sqrt{\alpha}\left(1-\alpha\right)Z$,
for $Z\sim\mathcal{N}\left(0,1\right)$ (CLT);
\item $\frac{\sqrt{n\bar{i}_{n}}}{n+1}\frac{i_{n}}{n+1}\frac{1}{\sqrt{\bar{i}_{n}}}\sum_{i=i_{n}+1}^{n+1}\left(\xi_{i}-1\right)\stackrel[\infty]{n}{\Longrightarrow}\alpha\sqrt{1-\alpha}Z'$,
for $Z'\sim\mathcal{N}\left(0,1\right)$ (CLT);
\item $\sum_{i=1}^{i_{n}}\xi_{i}\bot\sum_{i=i_{n}+1}^{n+1}\xi_{i}$ implies
$Z\bot Z'$, where $\bot$ indicates independence;
\item Numerator of (\ref{eq:combined}) $\stackrel[\infty]{n}{\Longrightarrow}\sqrt{\alpha}\left(1-\alpha\right)Z-\alpha\sqrt{1-\alpha}Z'$
(independence);
\item $\sqrt{\alpha}\left(1-\alpha\right)Z-\alpha\sqrt{1-\alpha}Z'\sim\mathcal{N}\left(0,\alpha\left(1-\alpha\right)\right)$.
\end{enumerate}
While the numerator of (\ref{eq:combined}) converges to $\mathcal{N}\left(0,\alpha\left(1-\alpha\right)\right)$
and the denominator converges almost surely to one, Slutsky's theorem
gives the desired result.
\end{proof}
\begin{lem}
\label{lem:Marg-Lik}$X\sim\mathcal{N}\left(0,\sigma^{2}+\tau^{2}\right)\textrm{ and }\left.X\right|\Theta\sim\mathcal{N}\left(\Theta,\tau^{2}\right)\implies\Theta\sim\mathcal{N}\left(0,\sigma^{2}\right)$.
\end{lem}

\begin{proof}
Let $f\left(x\right)$, $p\left(\left.x\right|\theta\right)$, and
$g\left(\theta\right)$ be the densities of $X$, $\left.X\right|\Theta$,
and $\Theta$. Then,
\begin{equation}
f\left(x\right)=\int_{-\infty}^{\infty}p\left(\left.x\right|\theta\right)g\left(\theta\right)d\theta.\label{eq:integral}
\end{equation}
Taking the Fourier transform of both sides of (\ref{eq:integral})
yields
\begin{align}
\hat{f}\left(t\right) & \coloneqq\int_{-\infty}^{\infty}f\left(x\right)\exp\left(itx\right)dx=\exp\left(\nicefrac{-t^{2}\left(\sigma^{2}+\tau^{2}\right)}{2}\right)\\
 & =\int_{-\infty}^{\infty}\int_{-\infty}^{\infty}p\left(\left.x\right|\theta\right)g\left(\theta\right)d\theta\exp\left(itx\right)dx\\
 & =\int_{-\infty}^{\infty}\int_{-\infty}^{\infty}p\left(\left.x\right|\theta\right)\exp\left(itx\right)dx\,g\left(\theta\right)d\theta\label{eq:Fubini}\\
 & =\exp\left(\nicefrac{-t^{2}\tau^{2}}{2}\right)\int_{-\infty}^{\infty}g\left(\theta\right)\exp\left(it\theta\right)d\theta\\
 & \eqqcolon\exp\left(\nicefrac{-t^{2}\tau^{2}}{2}\right)\hat{g}\left(t\right),
\end{align}
where (\ref{eq:Fubini}) uses Fubini's theorem. Taking the inverse
Fourier transform gives 
\begin{align*}
g\left(\theta\right) & =\int_{-\infty}^{\infty}\hat{g}\left(t\right)\exp\left(-it\theta\right)dt=\int_{-\infty}^{\infty}\exp\left(\nicefrac{-t^{2}\sigma^{2}}{2}\right)\exp\left(-it\theta\right)dt\\
 & =\left.\exp\left(\nicefrac{-1}{2}\left(\nicefrac{\theta}{\sigma}\right)^{2}\right)\right/\sqrt{2\pi\sigma^{2}}.
\end{align*}
\end{proof}
\remedRankToNormal*
\begin{proof}
With $R_{1,b}\equiv\nicefrac{\left(b+1\right)}{2}$ the result holds
when $k=1$, so assume that $k\geq2$. For $U_{1},\ldots,U_{b^{k}}\stackrel{\mathrm{iid}}{\sim}\mathrm{Uniform}\left(0,1\right)$,
Theorem \ref{thm:rem_to_normal} gives
\begin{equation}
2b^{\nicefrac{k}{2}}\left(U_{\left(R_{k,b}\right)}-\nicefrac{1}{2}\right)\stackrel[\infty]{b}{\Longrightarrow}\mathcal{N}\left(0,\left(\nicefrac{\pi}{2}\right)^{k-1}\right).\label{eq:limUbar}
\end{equation}
For $r\in\mathbb{R}$ and 
\begin{align*}
\mathcal{R}_{k,b} & \coloneqq\left\{ \left\lceil \nicefrac{b}{2}\right\rceil ^{k},\left\lceil \nicefrac{b}{2}\right\rceil ^{k}+1,\ldots,b^{k}+1-\left\lceil \nicefrac{b}{2}\right\rceil ^{k}\right\} \\
\overline{\mathcal{R}}_{k,b} & \coloneqq2b^{\nicefrac{k}{2}}\left(b^{-k}\mathcal{R}_{k,b}-\nicefrac{1}{2}\right),
\end{align*}
let $r_{j}\in\overline{\mathcal{R}}_{k,b_{j}}$ be a sequence of values
such that $b_{j}=b_{0}+2j$, for $b_{0}\in\mathcal{B}_{3}$, and $\lim_{j\rightarrow\infty}r_{j}=r$
(see Lemma \ref{lem:existLim}), and let 
\begin{equation}
\overline{U}_{k,j}\coloneqq2b_{j}^{\nicefrac{k}{2}}\left(U_{\left(R_{k,b_{j}}\right)}-\nicefrac{1}{2}\right)\textrm{ and }\overline{R}_{k,j}\coloneqq2b_{j}^{\nicefrac{k}{2}}\left(b_{j}^{-k}R_{k,b_{j}}-\nicefrac{1}{2}\right).
\end{equation}
Then, $\overline{R}_{k,j}=r_{j}$ implies that $R_{k,b_{j}}=\frac{b_{j}^{k}+b_{j}^{\nicefrac{k}{2}}r_{j}}{2}$
and 
\begin{align*}
\mathrm{E}\left[\left.\overline{U}_{k,j}\right|\overline{R}_{k,j}=r_{j}\right] & =2b_{j}^{\nicefrac{k}{2}}\left(\frac{b_{j}^{k}+b_{j}^{\nicefrac{k}{2}}r_{j}}{2\left(b_{j}^{k}+1\right)}-\frac{1}{2}\right)\stackrel{j\rightarrow\infty}{\longrightarrow}r\\
\mathrm{Var}\left(\left.\overline{U}_{k,j}\right|\overline{R}_{k,j}=r_{j}\right) & =4b_{j}^{k}\frac{\left(b_{j}^{k}+b_{j}^{\nicefrac{k}{2}}r_{j}\right)\left(b_{j}^{k}-b_{j}^{\nicefrac{k}{2}}r_{j}+2\right)}{4\left(b_{j}^{k}+1\right)^{2}\left(b_{j}^{k}+2\right)}\stackrel{j\rightarrow\infty}{\longrightarrow}1
\end{align*}
because $\left(\left.U_{\left(R_{k,b}\right)}\right|R_{k,b}\right)\sim\mathrm{Beta}\left(R_{k,b},b^{k}+1-R_{k,b}\right)$.
Then, Lemma \ref{lem:DirK2} \footnote{Use 
\begin{align*}
\overline{U}_{k,j} & =2b_{j}^{\nicefrac{k}{2}}\left(U_{\left(R_{k,b_{j}}\right)}-\frac{b_{j}^{k}+b_{j}^{\nicefrac{k}{2}}r_{j}}{2\left(b_{j}^{k}+1\right)}\right)+2b_{j}^{\nicefrac{k}{2}}\left(\frac{b_{j}^{k}+b_{j}^{\nicefrac{k}{2}}r_{j}}{2\left(b_{j}^{k}+1\right)}-\frac{1}{2}\right)\\
 & =2b_{j}^{\nicefrac{k}{2}}\left(U_{\left(R_{k,b_{j}}\right)}-\frac{b_{j}^{k}+b_{j}^{\nicefrac{k}{2}}r_{j}}{2\left(b_{j}^{k}+1\right)}\right)+\frac{b_{j}^{k}r_{j}}{b_{j}^{k}+1}+b_{j}^{\nicefrac{k}{2}}\left(\frac{b_{j}^{k}}{b_{j}^{k}+1}-1\right)
\end{align*}
 and Slutsky's theorem.} gives
\begin{equation}
\left(\left.\overline{U}_{k,j}\right|\overline{R}_{k,j}=r_{j}\right)\stackrel[\infty]{j}{\Longrightarrow}\mathcal{N}\left(r,1\right).\label{eq:limUbar|Xbar}
\end{equation}
Finally, panels (\ref{eq:limUbar}) through (\ref{eq:limUbar|Xbar})
and Lemma \ref{lem:Marg-Lik} give the desired result.
\end{proof}

\section{Proof of Theorem \ref{thm:rem_med_to_quad_normal} \label{sec:Multivariate-Convergence}}

We close by deriving the quadrivariate limit in Theorem \ref{thm:rem_med_to_quad_normal}.
Note that the limit for the last three variables holds when $\mathbb{E}\left|X_{1}\right|=\infty$.

\remMedToQuadNormal*
\begin{proof}
This is true for $k=1$. We assume $k\geq2$. By the Cram\'er-Wold theorem
(\citet{B95}, page 383), the result follows if we can show that $V_{k,b}\implies\mathcal{N}\left(0,\:\mathbf{a}^{\mathsf{T}}\boldsymbol{\Sigma}\mathbf{a}\right),$
as $b\rightarrow\infty$, where $\mathbf{a}\in\mathbb{R}^{4}$, 
\begin{equation}
\boldsymbol{\Sigma}\coloneqq\left(\begin{array}{cccc}
\sigma^{2} & \eta & \eta & 0\\
\eta & 1 & 1 & 0\\
\eta & 1 & \tau_{k}^{2} & \bar{\tau}_{k}^{2}\\
0 & 0 & \bar{\tau}_{k}^{2} & \bar{\tau}_{k}^{2}
\end{array}\right),\label{eq:Sigma}
\end{equation}
and
\begin{multline*}
V_{k,b}\coloneqq a_{1}b^{\nicefrac{k}{2}}\left(\bar{X}_{b^{k}}-\mu\right)+2a_{2}f\left(\bar{\mu}\right)b^{\nicefrac{k}{2}}\left(X_{\left(h_{k,b}\right)}-\bar{\mu}\right)\\
+2a_{3}f\left(\bar{\mu}\right)b^{\nicefrac{k}{2}}\left(X_{\left(R_{k,b}\right)}-\bar{\mu}\right)+2a_{4}b^{\nicefrac{k}{2}}\left(b^{-k}R_{k,b}-\nicefrac{1}{2}\right).
\end{multline*}
For $r\in\mathbb{R}$ and 
\begin{align*}
\mathcal{R}_{k,b} & \coloneqq\left\{ \left\lceil \nicefrac{b}{2}\right\rceil ^{k},\left\lceil \nicefrac{b}{2}\right\rceil ^{k}+1,\ldots,b^{k}+1-\left\lceil \nicefrac{b}{2}\right\rceil ^{k}\right\} \\
\overline{\mathcal{R}}_{k,b} & \coloneqq2b^{\nicefrac{k}{2}}\left(b^{-k}\mathcal{R}_{k,b}-\nicefrac{1}{2}\right),
\end{align*}
let $r_{j}\in\overline{\mathcal{R}}_{k,b_{j}}$ be a sequence of values
such that $b_{j}=b_{0}+2j$, for $b_{0}\in\mathcal{B}_{3}$, and $\lim_{j\rightarrow\infty}r_{j}=r$
(see Lemma \ref{lem:existLim}). Further, let
\[
\begin{array}{lcl}
\overline{R}_{k,j}\coloneqq2b_{j}^{\nicefrac{k}{2}}\left(b_{j}^{-k}R_{k,b_{j}}-\nicefrac{1}{2}\right) &  & \hspace{1bp}Z_{j}^{\left(h\right)}\coloneqq\sum_{i=1}^{b_{j}^{k}}\mathbf{1}_{\left\{ X_{i}>\bar{\mu}\right\} }\\
\hspace{10bp}p_{j}\coloneqq\frac{1+r_{j}b_{j}^{\nicefrac{-k}{2}}}{2}\eqqcolon1-q_{j} &  & Z_{j}^{\left(R\right)}\coloneqq\sum_{i=1}^{b_{j}^{k}}\mathbf{1}_{\left\{ X_{i}>F^{-1}\left(p_{j}\right)\right\} }.
\end{array}
\]
As $j\rightarrow\infty$, Lemma \ref{lem:bahadur66} then gives $\left(\left.V_{k,b_{j}}\right|\overline{R}_{k,j}=r_{j}\right)$
\begin{align}
 & \stackrel{\mathrm{a.s.}}{\sim}a_{1}b_{j}^{\nicefrac{k}{2}}\left(\bar{X}_{b_{j}^{k}}-\mu\right)+2a_{2}\frac{Z_{j}^{\left(h\right)}-\nicefrac{b_{j}^{k}}{2}}{b_{j}^{\nicefrac{k}{2}}}+2a_{3}\frac{Z_{j}^{\left(R\right)}-b_{j}^{k}q_{j}}{b_{j}^{\nicefrac{k}{2}}}+a_{4}r_{j}\nonumber \\
 & =\frac{\sum_{i=1}^{b_{j}^{k}}\left\{ \left(a_{1}X_{i}+2a_{2}\mathbf{1}_{\left\{ X_{i}>\bar{\mu}\right\} }+2a_{3}\mathbf{1}_{\left\{ X_{i}>F^{-1}\left(p_{j}\right)\right\} }\right)-\left(a_{1}\mu+a_{2}+a_{3}\right)\right\} }{b_{j}^{\nicefrac{k}{2}}}+\left(a_{3}+a_{4}\right)r_{j}\nonumber \\
 & \implies\mathcal{N}\left(\left(a_{3}+a_{4}\right)r,\:\mathrm{Var}\left(a_{1}X_{1}+2\left(a_{2}+a_{3}\right)\mathbf{1}_{\left\{ X_{1}>\bar{\mu}\right\} }\right)\right)\label{eq:useCLT}\\
 & =\mathcal{N}\left(\left(a_{3}+a_{4}\right)r,\:a_{1}^{2}\sigma^{2}+\left(a_{2}+a_{3}\right)^{2}+2a_{1}\left(a_{2}+a_{3}\right)\eta\right),\label{eq:getVar}
\end{align}
where (\ref{eq:useCLT}) uses the CLT, Slutsky's theorem, $q_{j}\rightarrow\nicefrac{1}{2}$,
and (\ref{eq:AssumptionCOI_half}), and (\ref{eq:getVar}) uses
\begin{align}
2\,\mathrm{Cov}\left(X_{1},\mathbf{1}_{\left\{ X_{1}>\bar{\mu}\right\} }\right) & =2\int_{\bar{\mu}}^{\infty}xdF\left(x\right)-\mu\label{eq:2Cov}\\
 & =2\int_{\bar{\mu}}^{\infty}xdF\left(x\right)-\left(\int_{-\infty}^{\bar{\mu}}xdF\left(x\right)+\int_{\bar{\mu}}^{\infty}xdF\left(x\right)\right)\nonumber \\
 & =\int_{\bar{\mu}}^{\infty}xdF\left(x\right)-\int_{-\infty}^{\bar{\mu}}xdF\left(x\right)+\bar{\mu}-\bar{\mu}\nonumber \\
 & =\int_{-\infty}^{\bar{\mu}}\left(\bar{\mu}-x\right)dF\left(x\right)+\int_{\bar{\mu}}^{\infty}\left(x-\bar{\mu}\right)dF\left(x\right)\nonumber \\
 & =\mathbb{E}\left|X_{1}-\bar{\mu}\right|\eqqcolon\eta.\nonumber 
\end{align}
To summarize, we have 
\begin{align}
\left(\left.V_{k}\right|\overline{R}_{k}=r\right) & \coloneqq\lim_{j\rightarrow\infty}\left(\left.V_{k,b_{j}}\right|\overline{R}_{k,j}=r_{j}\right)\sim\mathcal{N}\left(Ar,\:B\right)\label{eq:Vk-dist}\\
\textrm{and }\overline{R}_{k} & \coloneqq\lim_{j\rightarrow\infty}\overline{R}_{k,j}\sim\mathcal{N}\left(0,\:C\right),\label{eq:Rbark-dist}
\end{align}
where (\ref{eq:Vk-dist}) uses $A$ and $B$ defined by (\ref{eq:getVar}),
and (\ref{eq:Rbark-dist}) uses $C\coloneqq\left(\nicefrac{\pi}{2}\right)^{k-1}-1$
as in Theorem \ref{thm:rem_rank_to_normal}. Letting $V_{k}\coloneqq\lim_{b\rightarrow\infty}V_{k,b}$
and $v\in\mathbb{R}$, we finally have
\begin{align*}
\frac{d\Pr\left(V_{k}\leq v\right)}{dv} & =\int_{-\infty}^{\infty}\frac{d\Pr\left(\left.V_{k}\leq v\right|\overline{R}_{k}=r\right)}{dv}\frac{d\Pr\left(\overline{R}_{k}\leq r\right)}{dr}dr\\
 & =\frac{\int_{-\infty}^{\infty}\exp\left(-\frac{\left(v-Ar\right)^{2}}{2B}-\frac{r^{2}}{2C}\right)dr}{2\pi\sqrt{BC}}=\frac{\exp\left(-\frac{v^{2}}{2\left(A^{2}C+B\right)}\right)}{\sqrt{2\pi\left(A^{2}C+B\right)}},
\end{align*}
so that $V_{k}\sim\mathcal{N}\left(0,\:A^{2}C+B\right)=\mathcal{N}\left(0,\:\mathbf{a}^{\mathsf{T}}\boldsymbol{\Sigma}\mathbf{a}\right)$,
for $\boldsymbol{\Sigma}$ as in (\ref{eq:Sigma}), completing the
proof.
\end{proof}

\end{document}